\documentclass[11pt]{article}

\usepackage{lipsum}
\usepackage{amsfonts}
\usepackage{amssymb} 
\usepackage{graphicx}
\usepackage{epstopdf}
\usepackage{amsmath}
\usepackage[hypertexnames=false,colorlinks=true,linkcolor=blue,citecolor=blue]{hyperref}
\usepackage{cleveref}
\usepackage{algorithmic}

\ifpdf
  \DeclareGraphicsExtensions{.eps,.pdf,.png,.jpg}
\else
  \DeclareGraphicsExtensions{.eps}
\fi
\usepackage{subfig}

 \usepackage[top=2.5cm, bottom=2.5cm, left=2.5cm, right=2.5cm]{geometry}

\newtheorem{lemma}{Lemma}[section]
\newtheorem{theorem}{Theorem}

\newtheorem{remark}{Remark}

\newenvironment{proof}%
 {\begin{trivlist} \item[]{\bf Proof. }}%
 {\hspace*{\fill}$\rule{.4\baselineskip}{.4\baselineskip}$\end{trivlist}}

\usepackage{enumitem}
\setlist[enumerate]{leftmargin=.5in}
\setlist[itemize]{leftmargin=.5in}

\newcommand{\Mminus}{M_{\pi,\epsilon}^{-}}
\newcommand{\Mplus}{M_{\pi,\epsilon}^{+}}
\newcommand{\MF}{M_{F,\epsilon}}

\newcommand{\MminusZ}{M_{\pi,0}^{-}}
\newcommand{\MplusZ}{M_{\pi,0}^{+}}
\newcommand{\MFZ}{M_{F,0}}

\title{Slow-fast dynamics of strongly coupled adaptive frequency oscillators\thanks{
Part of this work was supported by New York University and the European Research Council (ERC) under the European Union's Horizon 2020 research and innovation programme (grant agreement No 637935) 
}}

\author{Ludovic Righetti\thanks{New York University (\tt{ludovic.righetti@nyu.edu})}
     \and Jonas Buchli\thanks{DeepMind, London, UK, (\tt{buchli@google.com})}
  \and Auke J. Ijspeert\thanks{Ecole Polytechnique F\'ed\'erale de Lausanne, Lausanne, Switzerland, (\tt{auke.ijspeert@epfl.ch})}
}

\usepackage{amsopn}

\begin{document}
\maketitle

\begin{abstract}
Oscillators have two main limitations: their synchronization properties are limited (i.e they have a finite synchronization region) and they have no memory of past interactions (i.e. they return to their intrinsic frequency whenever the entraining signal disappears).
We previously proposed a general mechanism to transform an oscillator into an adaptive frequency oscillator which adapts its parameters to learn the frequency of any input signal. The synchronization region then becomes infinite and the oscillator retains the entrainment frequency when the driving signal disappears.
While this mechanism has been successfully used in various applications, such as robot control or observer design for active prosthesis, a formal understanding of its properties is still missing.
In this paper, we study the adaptation mechanism in the case of strongly coupled phase oscillators
and show that non-trivial slow-fast dynamics is at the origin of the adaptation.
We show the existence of a layered structure of stable and unstable invariant slow manifolds and demonstrate how the input signal forces the dynamics to jump between these manifolds at regular intervals, leading to exponential convergence of the frequency adaptation. We extend the idea to a network of oscillators with amplitude adaptation and show that the slow invariant manifolds structure persists.
Numerical simulations validate our analysis and extend the discussion to more complex cases.
\end{abstract}

%\begin{keywords} 
%Oscillators, synchronization, frequency adaptation, slow-fast dynamics
%\end{keywords}

%\begin{AMS}
%34C15 34C60 37Nxx
%\end{AMS}

\section{Introduction}
Oscillators are used increasingly in science and engineering, either for modeling or design purposes.
They are well suited for applications that involve synchronization with periodic signals. 
However, since they traditionally have a fixed intrinsic frequency, two main limitations arise.
First their synchronization properties are limited in the sense that they can synchronize only with signals with close enough frequencies, i.e. they have a finite synchronization region.
Second, they have no memory of past interactions, i.e. if the entrainment signal disappears they return to their original frequency of oscillations.

Consequently when one wants to design systems that have unlimited synchronization capabilities and/or where past interactions (i.e. memory) plays an important role, these models are not well adapted.
Some biological oscillators appear to have a mechanism to adapt their intrinsic frequencies, for example to explain the synchronization phenomena of some species of fireflies \cite{ermentrout91} or to explain how the neural pattern generators that control the locomotion of animals can adapt to a body that changes dramatically in size during the development of the animal \cite{ijspeert08}.
In engineering applications, it can be beneficial to have systems capable to synchronize to unknown, noisy and potentially time-varying periodic inputs without the need to consider synchronization regions. For example, to ensure a controller automatically adapts to the gait-dependent resonant frequencies of a robot \cite{buchli08b}).
Finally, dynamical systems memorizing frequencies of past interactions afford a simple form of learning.

In \cite{buchli04b,righetti06b}, we proposed a general mechanism to transform a nonlinear oscillator into an adaptive frequency oscillator, i.e. an oscillator that can adapt its parameters to learn the frequency of an arbitrary periodic input signal.
This mechanism was used, for example, to transform Hopf, Van der Pol, Rayleigh and Fitzugh-Nagumo oscillators and also the R\"ossler strange attractor into adaptive frequency systems \cite{righetti06b}. 
This effect goes beyond mere synchronization as it works for ranges of frequencies beyond usual synchronization regions (infinite range in the case of phase oscillators) and the adapted frequency remains even when the input signal disappears. Moreover, the oscillator can track changes in the frequency of the input.
This approach has been used in several
robotic, control and estimation applications over the past decade but has never been formally studied apart for the weak coupling case.

In this paper, we study the prototypical case of an adaptive frequency phase oscillator with strong coupling
\begin{align}
    \dot{\phi} &= \lambda\omega - K \sin\phi F(t)\label{eq:phaseafo_intro1}\\
    \dot{\omega} &= - K \sin\phi F(t)\label{eq:phaseafo_intro2}
\end{align}
where $\phi$ is the phase of the oscillator, $\omega$ its frequency, $\lambda>0$ a constant parameter and $K>0$ the coupling strength. This system adapts its frequency to the frequency of the external input $F(t)$. After adaptation,
the frequency of the input signal can explicitly be read out from $\omega$, i.e. the system can extract the frequency of a periodic input signal without assumptions on $F(t)$ or the need of an explicit Fourier transform. 

\subsection{Previous results for weak coupling}
We previously proved the convergence of $\omega$ to 
one of the frequency component of a periodic input $F(t)$ for small coupling $K\ll 1$ \cite{righetti06b}. Through perturbation analysis we showed that frequency adaptation was taking place at the second order perturbation, thus emphasizing the importance of the interaction between the tendency of the oscillator to synchronize and the dynamics of $\omega$, both having an evolution on two different time-scales.
In particular, we showed that for $F(t) = \sum_{n=-\infty}^\infty A_n \mathrm{e}^{i n \omega_F t}$, the frequency adaptation behaved locally as
\begin{align}
\omega(t) &= \omega_0 + KP(t) + K^2 D_\omega(t)+ O(K^3)\\
D_\omega(t) &= \left(\frac{-A_0}{2 \omega_0} + \sum_{n\in\mathrm{N}^*} \frac{|A_n|^2 \omega_0}{((n\omega_F)^2 - \omega_0^2)} \right) (t - t_0)\label{eq:small_coupling}
\end{align}
where $P(t)$ is periodic with 0 mean, $\omega_0$ and $t_0$ are the initial conditions of the system. This shows that $\omega(t)$ at second order
has a linear drift towards one of the frequency component $n\omega_F$ of $F(t)$, depending on the initial frequency of the oscillator. The results also provide an approximate characterization of the different basins of attraction (separated by the roots of $D_\omega$) for different frequencies.

While this analysis accurately describes the behavior of the system for weak coupling, it completely fails to capture the dynamics of the system for strong coupling which is the desirable mode of operation in real applications. Numerical simulations suggest that frequency adaptation persists for strong coupling $K\gg 1$ and after convergence, the frequency parameter oscillates around the correct frequency value with an amplitude bounded when $K\to \infty$. However, a rigorous analysis of the strong coupling case is still lacking. %This paper provides such an analysis.

\subsection{Networks of adaptive frequency oscillators}
In \cite{buchli08} we numerically studied the behavior of a large number of such adaptive phase frequency oscillators coupled via a negative mean field.
Numerical evidence showed that it was possible to very well extract the frequency spectrum of arbitrary signals in real-time, ranging from signals with discrete spectra to ones with time-varying and continuous spectra. 
One interesting observation was that for time-varying spectra, the ability of the oscillators to follow time-varying frequencies resembles a first order linear system with cutoff frequency at $1$ $\mathrm{rad}\cdot \mathrm{s}^{-2}$. It means that frequency change can be tracked well up to rates of change of $1$ $\mathrm{rad}\cdot \mathrm{s}^{-2}$. 
In this contribution we provide a rigorous explanation to this phenomenon.
This network of oscillators has been extended by adding an adaptive weight to each oscillator in the mean field sum \cite{righetti06}. The oscillator can then also adapt its amplitude to match the energy content of a specific frequency component of an input signal. This idea has been used, for example, to construct controllers that coordinate the joints of a legged robot during walking \cite{righetti06}.

\subsection{Applications to control and estimation}
Adaptive frequency oscillators have found numerous applications in control and estimation applications, especially robotics.
In adaptive control, they were used to automatically tune a controller to the resonant frequency of
a legged robot via a simple feedback loop \cite{buchli08b,buchli06b,buchli04b,buchli05}. 
In that case, the efficiency of the robot locomotion is automatically optimized and any change in the natural dynamics is tracked by the adaptive frequency oscillator without external intervention, which is especially useful when the robot changes gait.
Recently, the mechanism was used for the design of locomotion controllers that can quickly react to environmental changes \cite{nachstedt2017}.
These oscillators have also been used to estimate the temporal derivatives of periodic signals with no delay \cite{Ronsse:im}.
This type of estimation is part of the control system of an exoskeleton used in a robot-assisted rehabilitation context \cite{ronsse11}.
An electronic implementation of adaptive frequency oscillators was also proposed in \cite{ahmadi09}.

Networks of such oscillators were used to construct limit cycles in the context of robot learning from demonstration and robot control, where coupling between the oscillators was added to ensure stability when sensory feedback is added.
Such a network was originally used to learn a complex motion pattern from demonstrations and generate a controller capable of modulating online those patterns through feedback. It was applied to the control of bipedal locomotion in \cite{righetti06,righetti05b}. 
The idea was extended to learn and robustly generate other types of periodic movements for robots with arms and legs \cite{gams08b, gams09, Petric:2011dr}.

These and other applications using the adaptive frequency mechanism described above require a precise understanding of the properties and limits of the mechanism. This is particularly important for applications involving humans in the loop and safety-critical components. However, a formal analysis of this mechanism is still missing for strong coupling strengths, which is the most interesting mode of operation.

\subsection{Related work}
Oscillators that can adapt their frequency are not novel.
In \cite{ermentrout91}, a model for frequency adaptation of oscillators was proposed to model the synchronization behavior observed in fireflies. In \cite{borisyuk01, borisyuk_oscillatory_2004}, a network of second order phase oscillators were used to model novelty detection. In \cite{nakanishi03}, a second order phase oscillator was proposed to adapt the frequency of walking movement generation to the measured natural frequency of a biped robot and synchronize stepping.

Frequency adaptation has also been studied in networks of second-order phase oscillators, or oscillators with "inertial" effects, in Kuramoto-like models \cite{acebron98, acebron_synchronization_2000, tanaka_first_1997, taylor_spontaneous_2010}.
More general networks of adaptive dynamical systems, akin to the mechanism we proposed, have also been studied in \cite{Rodriguez:2014ey}, but with the use of dissipative coupling.

All of the models of frequency adaptation describe above assume that each oscillator has an explicit representation of either the phase of other oscillators, or the input signal's period or frequency. This is in contrast to the model we study, which makes no assumption on the nature of the input signal $F(t)$. This is particularly important for the robotic, control and estimation applications described above where exact properties of the input signals are not known in advance and typically change over time.

Closer to the model we study, \cite{nishii99} proposed a model of frequency adaptation for network of coupled phase oscillators that does not need knowledge of the input signal frequency or phase. However, the model necessitates the computation of time averages in the network limiting the ease of applicability of the approach.

\subsection{Contributions of the paper}
The need for a formal understanding of the frequency adaptation mechanism to support its safe deployment in robotic, control and estimation applications is the main motivation for this paper. 
We provides a complete description of the frequency adaptation mechanism for the strongly coupled phase oscillator model \cref{eq:phaseafo_intro1,eq:phaseafo_intro2}. 
Through geometric perturbation theory, we show that a slow-fast dynamics is responsible for exponential frequency adaptation and that the oscillator can extract frequency components of any periodic signal. We derive a map summarizing the slow-fast dynamics, which accurately describe the frequency adaptation mechanism for complex input signals. Further we show how the convergence rate can be controlled and the associated trade-offs in terms of convergence accuracy. 
Finally, we extends the analysis to networks of coupled oscillators augmented with amplitude adaptation. We provide a geometric characterization of the slow-fast dynamics and numerically investigate the behavior for complex input signals.
\footnote{The software used for the numerical simulations of this article is available as open source in \cite{codelink}}

\section{Geometric structure of frequency adaptation}\label{sec:2}

In this section, we derive results for the strong coupling case using the adaptive frequency phase oscillator
\begin{align}
\dot{\phi} &= \lambda\omega - K \sin(\phi) F(t) \label{eq:adapt_phase1}\\
\dot{\omega} &= - K \sin(\phi) F(t) \label{eq:adapt_phase2}
\end{align}
where we introduce the parameter $\lambda >0$, a term enabling the explicit control of the frequency adaptation convergence rate (as we will prove below).
Note that the frequency of the oscillator is $\lambda \omega$. $F(t)$ is a time varying input signal which we assume to be $C^\infty$.

We rewrite the problem
as a singular perturbation problem that can be tackled by geometric singular perturbation theory \cite{fenischel79,jones1609geometric}. By setting $K = \frac{1}{\epsilon}$ (where $\epsilon \ll 1$) and making the system autonomous, we obtain
\begin{align}
\epsilon\dot{\phi} &= \epsilon\lambda\omega - \sin(\phi) F(\theta) \label{eq:phase1}\\
\epsilon\dot{\omega} &= - \sin(\phi) F(\theta) \label{eq:phase2}\\
\dot{\theta} &= 1
\end{align}
We study the dynamics, first by characterising slow locally invariant manifolds, then by analyzing the fast dynamics and finally by deriving discrete maps describing the average dynamics.

\subsection{Slow dynamics and invariant slow manifolds}
We first aim to characterize invariant slow manifolds using Fenichel theorem \cite{fenischel79}. A central hypothesis to apply this theorem is that
the critical manifold, i.e. the fixed points to \cref{eq:phase1,eq:phase2}, be normally hyperbolic. This is equivalent to requiring that the linearization of the dynamics at each point on the manifold
has as many 0 eigenvalues as there are slow variables \cite{jones1609geometric}.
The Jacobian of the dynamics has two 0 eigenvalues and so any invariant slow manifold cannot be hyperbolic.
The situation can be changed though the change of coordinates $\Omega = \phi - \omega$ to get
\begin{align}
\epsilon{\dot{\omega}} &= - \sin(\Omega + \omega) F(\theta) \label{eq:slow_dyn1}\\
\dot{\Omega}  &= \lambda\omega \label{eq:slow_dyn2}\\
\dot{\theta} &= 1 \label{eq:slow_dyn3}
\end{align}
which is equivalent to the following fast system
\begin{align}
\omega^{'} &= - \sin(\Omega + \omega) F(\theta)\label{eq:fast_dyn1}\\
\Omega^{'}  &= \epsilon\lambda\omega \label{eq:fast_dyn2}\\
\theta^{'} &= \epsilon \label{eq:fast_dyn3}
\end{align}
where we re-scaled time as $t = \epsilon \tau$ and $'=\frac{d}{d\tau}$.

We can now characterize the invariant slow manifolds and
provide a first order approximation of the flow on these manifolds.
The main result is summarized in the following theorem.

\begin{theorem}\label{th:slow_manifolds}
For $\epsilon$ sufficiently small, there exist infinitely many slow locally invariant manifolds for the flow \cref{eq:slow_dyn1,eq:slow_dyn2,eq:slow_dyn3}. They consist of simply connected,
compact subsets of $\mathbb{R}^3$ such that 
\begin{equation}\label{eq:M_epsilon}
\omega = (k\pi - \Omega)\left( 1+\frac{\epsilon (-1)^k \lambda}{F(\theta)} \right) + O(\epsilon^2), \quad F(\theta) \neq 0, \quad k\in \mathbb{Z}
\end{equation}
The manifolds are attracting when $(-1)^{k+1}F(\theta)<0$ and repelling otherwise.
The slow flow on these manifolds is
\begin{align}
\omega &= -\Omega(0) \mathrm{e}^{-\lambda t} + O(\epsilon)\\
\Omega &= k\pi + \Omega(0) \mathrm{e}^{-\lambda t} + O(\epsilon)\\
\theta &= t
\end{align}
\end{theorem}

\begin{proof}
First we compute the critical manifolds when $\epsilon=0$. 
They will be such that $- \sin(\Omega + \omega) F(\theta)=0$. 
We can consider two cases, either $\Omega + \omega = k\pi$, $k\in \mathbb{Z}$ or $F(\theta)= 0$.
The Jacobian of the fast ODE at $\epsilon = 0$ is
\begin{equation}
J = \begin{bmatrix}
-\cos(\Omega + \omega)F(\theta) & -\cos(\Omega + \omega)F(\theta) & -\sin(\Omega + \omega) \frac{\partial F(\theta)}{\partial \theta} \\
0 & 0 & 0 \\
0 & 0 & 0
\end{bmatrix} \nonumber
\end{equation}
which taken at $F(\theta)=0$ has three zero eigenvalues so the corresponding critical manifold is not be hyperbolic.
Taken at $\Omega + \omega = k\pi$, the Jacobian has eigenvalues $0, 0$ and $(-1)^{k+1}F(\theta)$
with respective eigenvectors
\begin{equation}
\left( \begin{array}{c} -1 \\ 1 \\ 0 \end{array} \right)
\left( \begin{array}{c}  0  \\ 0 \\ 1 \end{array} \right)
\left( \begin{array}{c} 1 \\ 0 \\ 0 \end{array} \right) \nonumber
\end{equation}
so the direction transverse to the critical manifolds has a non-zero eigenvalue as long as $F(\theta)\neq 0$
and therefore the critical manifolds are hyperbolic.
The sign of $(-1)^{k+1}F(\theta)$ defines the attracting or repelling nature of the slow invariant manifolds.
These manifolds consist of any simply-connected, compact subsets of $\mathbb{R}^3$ such that
\begin{equation}
M_{0} \subset \{(\omega, \Omega, \theta) | \omega + \Omega = k\pi, F(\theta) \neq 0,\  k \in \mathbb{Z}\}
\end{equation}
We can invoke Fenichel theorem \cite{fenischel79,jones1609geometric} as the vector field is $C^\infty$ and the critical manifold is normally hyperbolic. We then conclude that for each manifold of type $M_{0}$ and for $\epsilon$ sufficiently small,
there exists a manifold $M_{\epsilon}$ that lies within $O(\epsilon)$ of $M_0$, that is diffeomorphic to $M_{0}$ and locally invariant to the flow of Equations \cref{eq:slow_dyn1}-\cref{eq:slow_dyn3}. The $\epsilon$ perturbation preserves
the attracting/repelling property of the manifold.

Using the characterization of the critical manifolds and the fast variable equation, 
\begin{equation}
\epsilon \dot{\omega}(\Omega,\theta,\epsilon) = -\sin(\Omega + \omega(\Omega,\theta,\epsilon)) F(\theta)
\end{equation}
we can write $\omega$ as a perturbation series in $\epsilon$ and match orders. Using at first order $\omega_0 = k\pi - \Omega$, direct computations show that
\begin{equation}
\omega = (k\pi - \Omega)\left( 1+\frac{\epsilon (-1)^k \lambda}{F(\theta)} \right) + O(\epsilon^2)
\end{equation}

The slow flow on these manifolds can be written
\begin{equation}
\dot{\Omega} = \lambda (k\pi - \Omega ) + O(\epsilon)
\end{equation}
which is linear at first order and therefore $\Omega(t) = k\pi + \Omega(0) e^{-\lambda t} + O(\epsilon)$ on the manifold.
Since $\Omega + \omega = k\pi + O(\epsilon)$ we also have $\omega(t) = - \Omega(0) e^{-\lambda t} + O(\epsilon)$, which finishes the proof. \qquad
\end{proof}

We have characterized the locally invariant slow manifolds of the system. In the $(\omega, \Omega)$ direction there is an alternation of attracting and repelling manifolds. In the $\theta$ direction, attracting and repelling invariant manifolds are separated by $F(\theta)=0$. Therefore, each time $F(\theta)$ changes sign, the flow in the neighborhood of an attracting manifold moves into the neighborhood of a repelling one and vice versa.

The $O(\epsilon)$ approximation of the slow manifold is important to understand where the flow exits an attracting manifold, with respect to a neighboring repelling manifold. Indeed, \cref{eq:M_epsilon} 
shows that for fixed $k$ on an attracting manifold we have $|\omega| > |k\pi - \Omega|$ and on
a repelling manifold $|\omega| < |k\pi - \Omega|$. Therefore, for a given $k$ and for $F(\theta)$ small, any pair of associated
attracting and repelling manifolds are on top of each other with always the same relative position as long as $k\pi - \Omega$
does not change sign.
When $F(\theta)$ changes sign, the flow on an attracting manifold is now in a neighborhood of a repelling one
and its position relative to the repelling one is always the same.
\cref{fig:manifolds} illustrates this interleaved structure.

A singular orbit on the critical manifolds is such that $\omega$ converges exponentially fast to $0$ with convergence rate $\lambda$ while $\Omega$ displays the same kind of convergence towards $k\pi$.
Moreover, on the critical manifolds the original system \cref{eq:adapt_phase1}-\cref{eq:adapt_phase2} is such that $\phi = k\pi + O(\epsilon)$, i.e. the slow dynamics is such that the phase of the oscillator is close to a constant.

\begin{figure}
\centering
\includegraphics[width=0.49\textwidth, trim=300 0 100 150, clip]{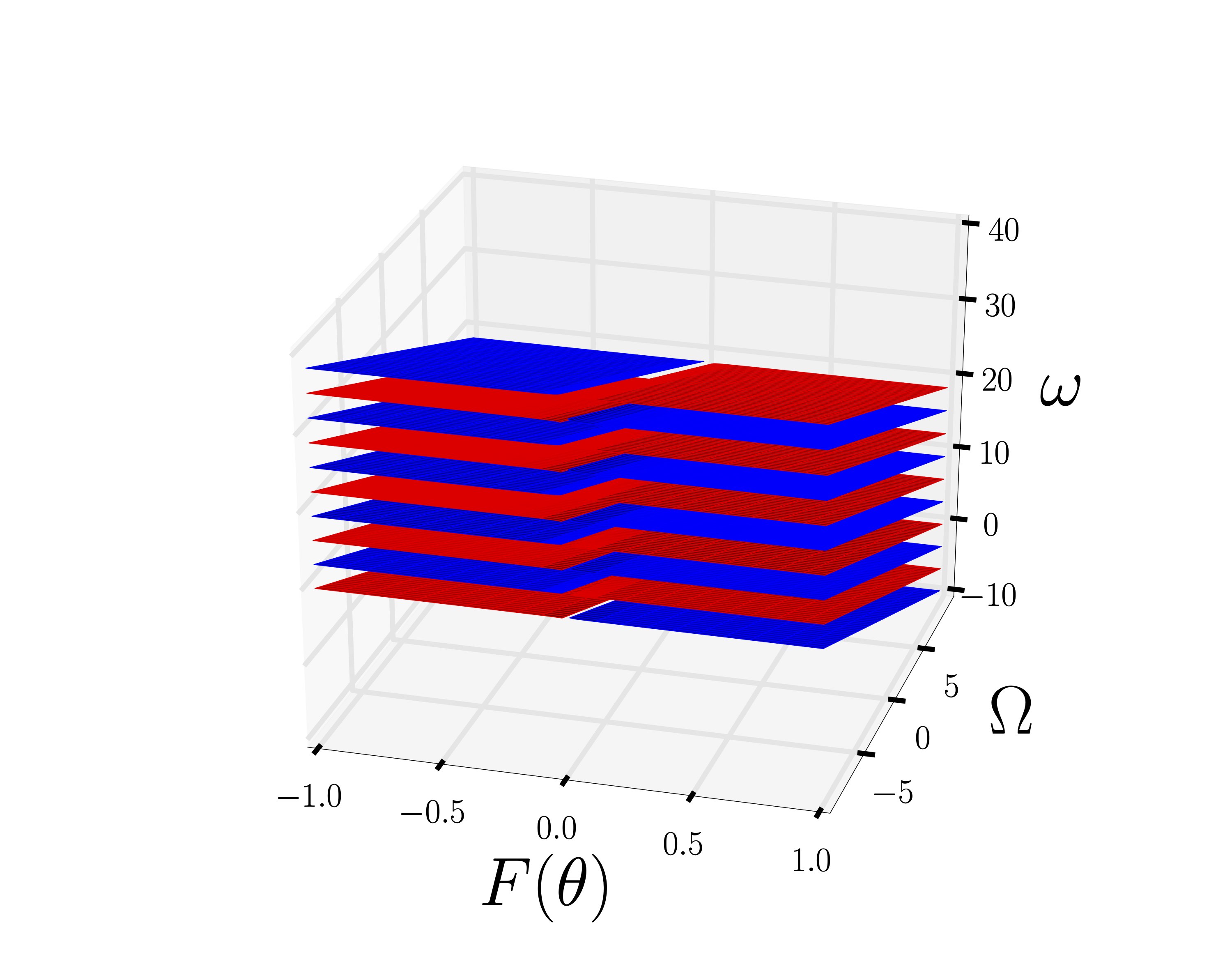}
\includegraphics[width=0.49\textwidth]{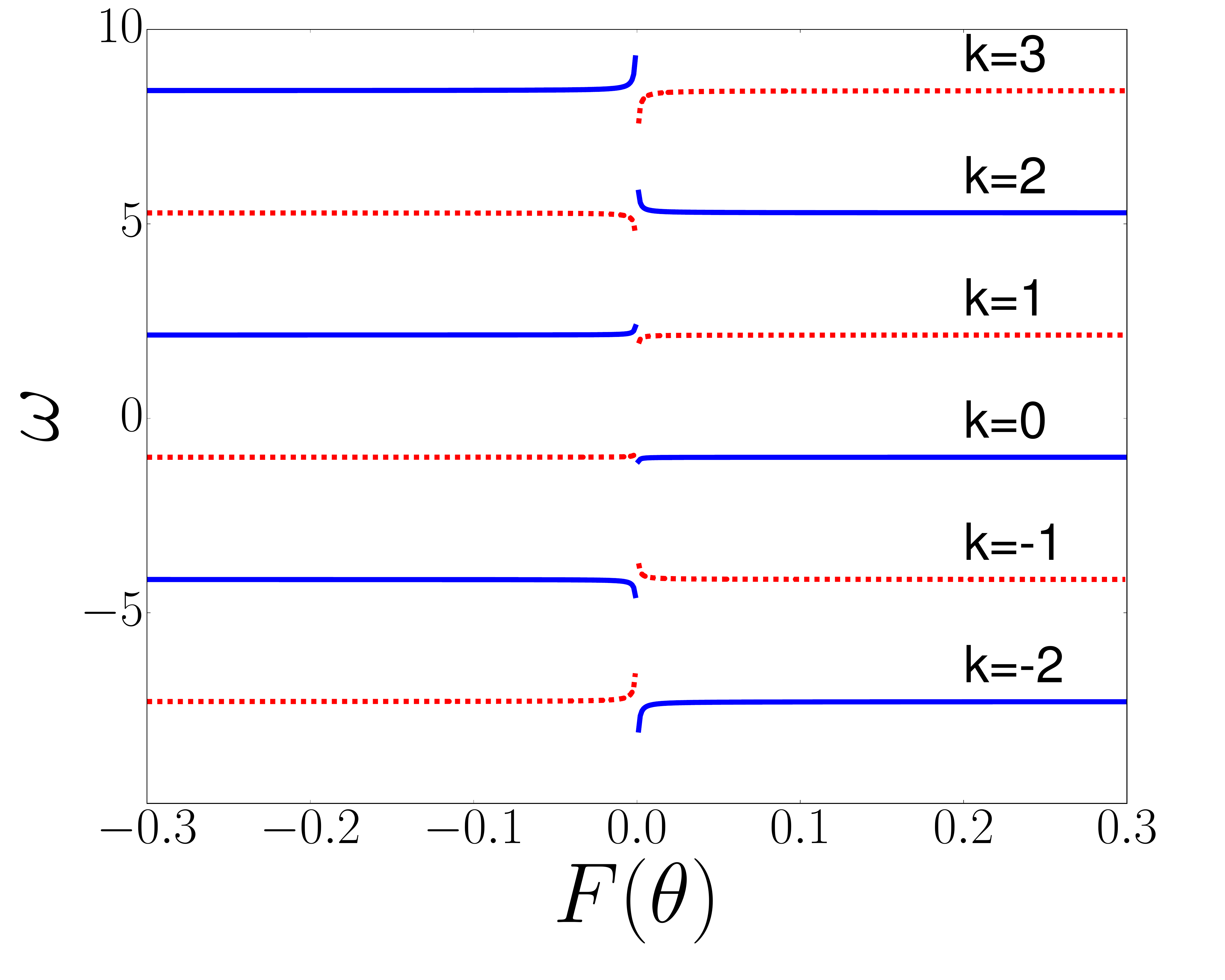}
\caption{$O(\epsilon)$ approximation of the slow invariant manifolds (repelling manifolds with dashed red lines and attracting ones with plain blue lines) as defined by \cref{eq:M_epsilon}.
The left graph shows a plot of $\omega$ as a function of $F(\theta)$ and $\Omega$ for $k = \{0,1,\cdots,10\}$, $\epsilon=10^{-4}$ and $\lambda=1$. We notice the alternation of repelling and attracting manifolds separated by $F(\theta)=0$. The right graph shows a 2D cut of the left graph for $\Omega=1$. On this graph we can see that in a neighborhood of $F(\theta)=0$ the attracting manifolds rapidly curve away from $\omega=0$ while the repelling ones curve towards it. This structure is important to understand how the flow of $\omega$ moves between attracting manifolds away from $\omega=0$ (e.g. when exiting a slow manifold when $F(\theta)$ changes sign, the flow is above the repelling manifold when $\omega>0$ and the fast flow will then converge to the attracting manifold on top of the repelling one).
}\label{fig:manifolds}
\end{figure}

\subsection{Fast dynamics}
The critical orbits of the fast dynamics are the solutions
of
\begin{align}
\omega' &= - \sin(\Omega + \omega) F(\theta) \\
\Omega'  &= 0 \\
\theta' &= 0
\end{align}
The fixed points of the fast dynamics correspond to the critical manifolds $M_0$, i.e. points of the form $\Omega + \omega = k\pi$. Such a fixed point is stable if the corresponding critical manifold is attracting, and unstable otherwise. We conclude that a critical orbit 
will then flow from the neighborhood of a repelling manifold to an attracting manifold. Due to the time reparametrization, the time scale associated
with the fast dynamics is controlled by the coupling constant $\epsilon$, meaning that the time taken to converge
to an attracting manifold is shorter as $\epsilon$ decreases. For the critical orbits, we assume that this happens instantaneously compared to the slow dynamics.

If the system starts in the neighborhood of a repelling manifold, the fast event makes it converge towards an attracting
manifold with a net variation of $\omega+\Omega$ of $\pm \pi$.
Because of the relative positions of the manifolds (cf. \cref{fig:manifolds}), when $F(\theta)$ changes sign the flow in a neighborhood of an attracting manifold moves to a neighborhood
of a repelling one such that $\omega + \Omega$ changes away from 0.
For example, in the case presented in \cref{fig:manifolds} (right graph),
$\omega+\Omega$ will increase by $\pi$ when $\omega>0$ and decrease by $\pi$ otherwise.
In the coordinates of the original system \cref{eq:adapt_phase1,eq:adapt_phase2}, it means that the phase $\phi$ will quickly change by $\pm \pi$.

\subsection{Convergence for periodic inputs}\label{sec:simple_cosine_analysis}
Thus far we characterized the singular orbits of the slow and fast  dynamics separately. We now piece these singular orbits together to explain how the succession of slow-fast events leads to an adaptive frequency mechanism. When $F(t)$ is periodic, the slow flow exits periodically the locally invariant manifold (i.e. each time $F(\theta)$ changes sign).
On or near an attracting manifold, the critical slow flow is such that $\omega$ converges exponentially fast to $0$ and close to a repelling invariant manifold, the critical fast orbit is such that $\omega$ increases by $\pi$.

\subsubsection{The case $F(t) = \cos(\omega_F t)$}
In the case of a simple periodic input $F(t) = \cos(\omega_F t)$, 
after one slow-fast event $\omega$ changes as
\begin{equation}
\omega_{n+1}^+ = \omega_n^+ \mathrm{e}^{-\frac{\lambda\pi}{\omega_F}} + \pi \label{eq:omega_discrete_p}
\end{equation}
where the duration of the slow event is assumed to be $t = \frac{\pi}{\omega_F}$, i.e. half of a period of the input. As we consider the critical orbits ($\epsilon=0$), we assume that the fast change of $\omega$ is instantaneous.
Similarly, the total change for $\omega$ after a succession of a fast and then a slow event is
\begin{equation}
\omega_{n+1}^- = (\pi + \omega_n^- )\mathrm{e}^{-\frac{\lambda\pi}{\omega_F}} \label{eq:omega_discrete_m}
\end{equation}
Both difference equations are linear and have only one globally stable fixed point
\begin{equation}
\bar{\omega}^+ = \frac{\pi}{1-\mathrm{e}^{-\frac{\lambda \pi}{\omega_F}}} 
\qquad
\bar{\omega}^- = \frac{\pi}{\mathrm{e}^{\frac{\lambda \pi}{\omega_F}}-1}
\end{equation}

\begin{lemma}
The fixed points $\bar{\omega}^+$ and $\bar{\omega}^-$ are globally asymptotically stable and are such that
$\omega_F < \lambda \bar{\omega}^+ < \omega_F + \lambda \pi $ and
$\omega_F -\lambda\pi < \lambda \bar{\omega}^- < \omega_F$.
Moreover, if $\frac{\lambda\pi}{2\omega_F} \ll 1$ (i.e. we assume a separation of time scales such that 
$F(t)$ changes sign at a faster rate
than the decay rate on the slow manifold),
the average $\tilde{\omega} = \frac{\bar{\omega}^+ + \bar{\omega}^-}{2}$
is
\begin{equation}
\lambda\tilde{\omega} \simeq \omega_F + O(\frac{\lambda\pi}{2\omega_F})
\end{equation}
\end{lemma}
\begin{proof}
Globally stability is direct since we have linear maps with a contracting coefficient $\mathrm{e}^{-\frac{\lambda\pi}{\omega_F}} < 1$.
By using the fact that $x+1 < \mathrm{e}^x < \frac{1}{1-x}$ when $x<1$ we find that 
$\omega_F < \lambda \bar{\omega}^+ < \omega_F + \lambda \pi $.
This directly leads to 
$\omega_F -\lambda\pi < \lambda \bar{\omega}^- < \omega_F$.
We have
\begin{equation}
\tilde{\omega} = \frac{\bar{\omega}^+ + \bar{\omega}^-}{2} = \frac{\pi}{2} \coth(\frac{\lambda\pi}{2\omega_F})
\end{equation}
Assuming that $\frac{\lambda\pi}{2\omega_F} \ll 1$, the series expansion of $\coth$ leads at first order to $\lambda\tilde{\omega} = \omega_F + O(\frac{\lambda\pi}{2\omega_F})$
\end{proof}

It is remarkable that this succession of slow-fast events
leads to an exponential convergence of $\lambda\omega$ to a neighborhood of $\omega_F$,
when there is a clear separation of time scale between the frequency of the input
$F(t)$ and the convergence rate $\lambda$.
Indeed, the system does not have explicit access to $\omega_F$ and it is really the timing of these events due to the zero-crossing of the input that induces convergence to the frequency of the input.
After convergence, $\omega$ oscillates at a frequency of $2\omega_F$ which is the frequency at which $F(t)$ changes sign.
The amplitude of oscillation for $\omega$ is $\pi$ which is the amount of change during a fast event (i.e.
$\bar{\omega}^+ - \bar{\omega}^- = \pi$).
Therefore, the precision at which we can recover $\omega_F$ from $\omega$ depends on the choice of the convergence rate $\lambda$ and
the bounds for the precision are of the form $\lambda\pi$.
Since each step in the difference equations corresponds to the evolution of $\omega$ for
a time $t = \frac{\pi}{\omega_F}$, the average convergence of $\omega$ will be of the form
\begin{equation}
\omega(t) \simeq (\omega - \tilde{\omega})\mathrm{e}^{-\lambda t} + \tilde{\omega} \label{eq:exp_convergence}
\end{equation}

\Cref{fig:convergence_prediction1} shows a typical evolution of the adaptive frequency oscillator with simple periodic input.
The figure also shows the very good correspondence of the predicted bounds $\omega_n^+$ and $\omega_n^-$ (\cref{eq:omega_discrete_p,eq:omega_discrete_m}) derived from the critical orbits with
the real evolution of $\omega$, as well the average exponential convergence \cref{eq:exp_convergence}.
In particular we see that $\lambda \tilde{\omega} = 100.008$ is very close to $\omega_F$. We note also that the continuous exponential convergence prediction, \cref{eq:exp_convergence}, gives a good approximation of the average dynamics.

\begin{figure}
\centering
\includegraphics[width=0.75\textwidth, trim=50 0 50 0, clip]{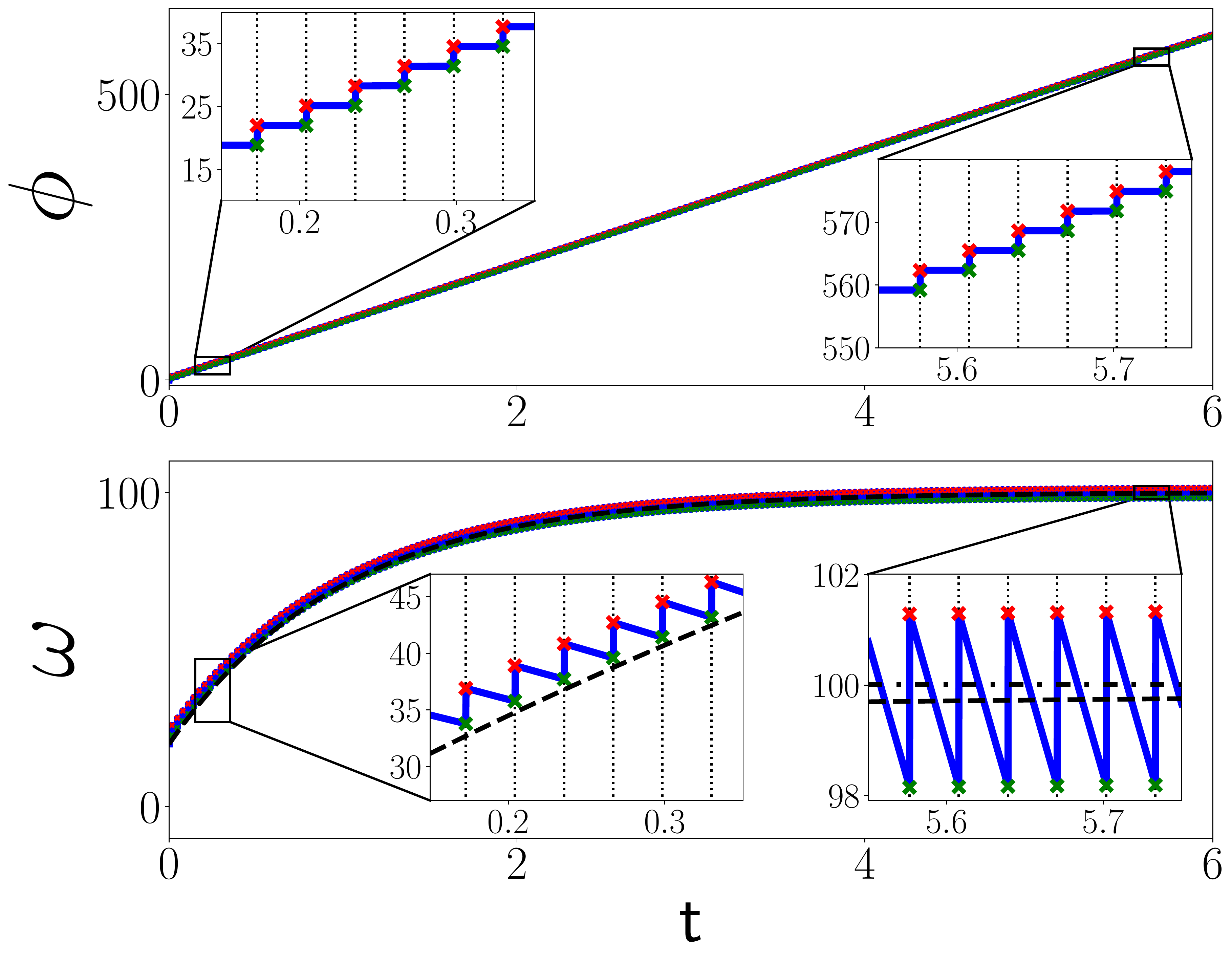}
\caption{Time evolution of the adaptive frequency oscillator for $F(t) = \cos(\omega_F t)$, where $\lambda =1$, $K=10^7$, $\omega_F=100$. The upper graph shows the evolution of $\phi(t)$ (blue line) and the prediction $\phi_n^+$ (red cross) and $\phi_n^-$ (green cross). In the zoomed plots, the vertical dashed lines correspond to $F(t) = 0$.
The lower graph shows the evolution of $\omega(t)$ (blue line), the predictions $\omega_n^+$ (red cross) an $\omega_n^-$ (green cross) as well as $(\omega - \tilde{\omega}) \mathrm{e}^{-\lambda t} + \tilde{\omega}$ (dashed line). In the zoomed plots, the vertical dashed lines correspond to $F(t) = 0$. In addition, the right zoomed plot shows $\tilde{\omega} = 100.008$ (dashed-dotted line).}\label{fig:convergence_prediction1}
\end{figure}

\subsubsection{General periodic functions}
We can use a similar analysis for more general functions. 
Indeed, we can predict $\omega$ after one slow-fast event (i.e. one zero-crossing of an arbitrary $F(t)$) with the discrete maps
\begin{align}
\omega^+_{n+1} &= \omega^+_n \mathrm{e}^{-\lambda \Delta_{t_i}} + \pi \label{eq:discrete_maps_ti1}\\
\omega^-_{n+1} &= (\pi+\omega^+_n) \mathrm{e}^{-\lambda \Delta_{t_i}}\label{eq:discrete_maps_ti2}
\end{align}
where $\Delta_{t_i}$ is the time between two zero crossings of $F(t)$. 

Without loss of generality, let's assume a periodic function
$F(t)$ of period $\frac{2\pi}{\omega_F}$ such that $F(t=0)=0$. Let's denote $t_i$, $i = 1 \cdots N$ the instants for which
the function is zero during one period (i.e. $t_i >0$ and $t_N = \frac{2\pi}{\omega_F}$). 
Combining the $N$ slow-fast events described by the maps \cref{eq:discrete_maps_ti1,eq:discrete_maps_ti2}, we can write two discrete linear maps of the evolution of $\omega$ after $N$ slow-fast events, i.e. a complete period of the input
\begin{align}
\hat{\omega}_{n+1}^+ &= \hat{\omega}_n^+ \mathrm{e}^{-\lambda \frac{2\pi}{\omega_F}} + \pi \mathrm{e}^{-\lambda \frac{2\pi}{\omega_F}} \sum_{i=1}^N \mathrm{e}^{\lambda t_i}\\
\hat{\omega}_{n+1}^- &= \hat{\omega}_n^- \mathrm{e}^{-\lambda \frac{2\pi}{\omega_F}} + \pi \mathrm{e}^{-\lambda \frac{2\pi}{\omega_F}} 
\left( 1 + \sum_{i=1}^{N-1} \mathrm{e}^{\lambda t_i} \right)
\end{align}
whose respective unique exponentially stable fixed points are
\begin{align}
\bar{\omega}^+ &= \frac{\pi}{\mathrm{e}^{\lambda \frac{2\pi}{\omega_F}}-1} \sum_{i=1}^N \mathrm{e}^{\lambda t_i}\\
\bar{\omega}^- &= \frac{\pi}{\mathrm{e}^{\lambda \frac{2\pi}{\omega_F}}-1} \left(1 + \sum_{i=1}^{N-1} \mathrm{e}^{\lambda t_i} \right)
\end{align}
As for the case of a simple cosine, we notice the exponential convergence towards the fixed points with
 rate controlled by $\lambda$. As before, we have $\bar{\omega}^+ - \bar{\omega}^- = \pi$.
 Note that these maps describe the evolution of $\omega$ over a complete period of $F(t)$, including several slow-fast events and that the actual dynamics of $\omega$ is not necessarily bounded between these two maps. The maps describing one slow-fast event \cref{eq:discrete_maps_ti1,eq:discrete_maps_ti2} need to be used instead if one needs to compute bounds on $\omega$.
The maps $\hat{\omega}_n^+$ and $\hat{\omega}_n^-$ however enable to derive the following result.
\begin{lemma}
The average $\tilde{\omega} = \frac{\bar{\omega}^+ + \bar{\omega}^-}{2}$ is such that 
\begin{equation}
\lim_{\lambda\to 0} \lambda \tilde{\omega} = \omega_F \frac{N}{2}
\end{equation}
where $N$ is the number of times the input signal changes sign over one period.
\end{lemma}
\begin{proof}
The average is
\begin{equation}
\tilde{\omega} = \frac{\bar{\omega}^+ + \bar{\omega}^-}{2} 
= \frac{\pi}{2(\mathrm{e}^{\lambda\frac{2\pi}{\omega_F}}-1)} \left(1 + \mathrm{e}^{\lambda\frac{2\pi}{\omega_F}}+2\sum_{i=1}^{N-1}\mathrm{e}^{\lambda t_i} \right)
\end{equation}
The rest of the proof follows from L'H\^opital's rule.
\end{proof}

In \Cref{sec:simple_cosine_analysis} we saw that $\lambda$ was a term that not only controlled the convergence
rate but also the precision at which $\omega_F$ could be recovered. $\lambda \to 0$ can be interpreted as the limit
for the best recovery of the input frequency. If $F(t)$ is periodic and continuous then $N$ is even as long as the zeros of $F(t)$ are not one of its extrema. Therefore,
when $\lambda\to 0$, the adaptive frequency oscillator will converge to an integer multiple of the fundamental frequency 
depending on the number of zeros of $F(t)$.

To illustrate this result, we simulate the system with the periodic input $F(t) = 1.3\cos(30t$\\ $+ 0.4) + \cos(60t) + 1.4\cos(90t +1.3)$ which has 4 zeros. Our lemma predicts that $\omega$ should converge in a neighborhood of $60$ (and not towards the frequency of the input which is $30$). \Cref{fig:conv_predict_period3} shows the results of the simulation, which confirms the prediction. 
Simulation results also show that every slow fast event can be accurately predicted using the discrete maps \cref{eq:discrete_maps_ti1,eq:discrete_maps_ti2}, suggesting that the 
description of the dynamics using the critical orbits is sufficient to capture the main features of the dynamical behavior of the system.

\begin{figure}
\centering
\includegraphics[width=0.75\textwidth, trim=50 0 50 0, clip]{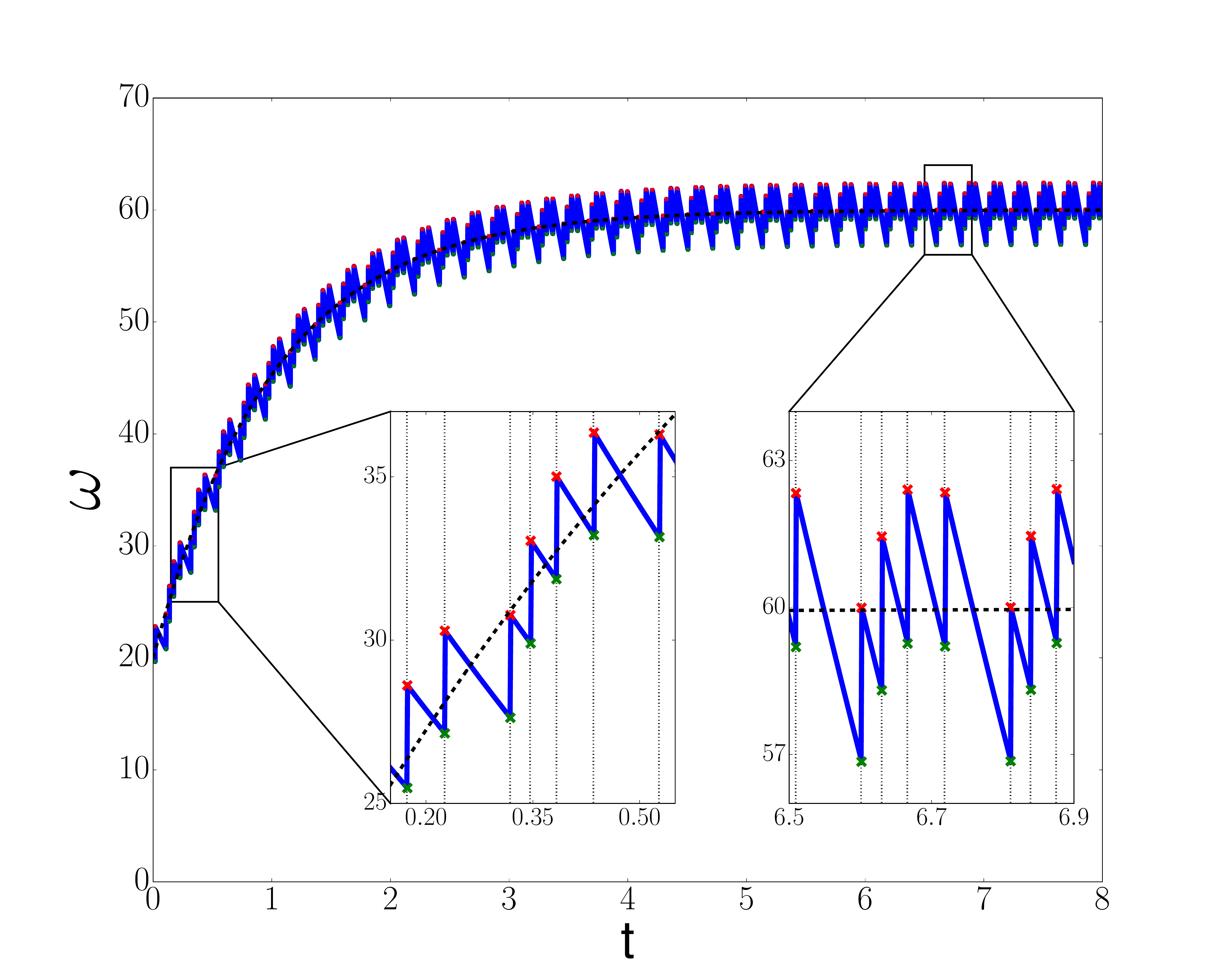}
\caption{Time evolution of the adaptive frequency oscillator for $F(t) = 1.3\cos(30t + 0.4) + \cos(60t) + 1.4\cos(90t +1.3)$, where $\lambda =1$, $K=10^6$, $\omega_F=30$. The graph shows the evolution of $\omega(t)$ (blue line), the predictions $\omega_n^+$ (red cross) an $\omega_n^-$ (green cross) taken from \cref{eq:discrete_maps_ti1,eq:discrete_maps_ti2} as well as $(\omega - \tilde{\omega}) \mathrm{e}^{-\lambda t} + \tilde{\omega}$ (dashed line). In the zoomed plots, the vertical dashed lines correspond to $F(t) = 0$.}\label{fig:conv_predict_period3}
\end{figure}

\subsubsection{Strictly positive inputs}
The frequency adaptation mechanism relies on the interaction between slow and fast dynamics
and it is driven by the sign changes of the external input.
Therefore, if $F(t)$ is a periodic signal that never changes sign then we expect no frequency adaptation.
After a transient, the flow will be on an attracting slow manifold and $\omega(t) \to 0$ (i.e. the oscillator will adapt to the "zero frequency" of the DC bias).
This is an important remark for real applications as input signals need to be processed to ensure appropriate sign changes, for example by removing the mean value of the input.

\section{Empirical evaluations with more complex inputs}
Our analysis thus far only considered period signals and strong coupling.
Here, we numerically investigate the system dynamics with aperiodic and chaotic inputs, time varying frequencies and reduced $K$.

\subsection{Effect of reduced coupling}
In practical applications, coupling might remain small for numerical
stability reasons. We explore whether exponential convergence remains possible.
We performed simulations of a phase oscillator coupled to a cosine input and evaluated its synchronization region (without frequency adaptation) for a large range of $K$ (between $1$ to $10^3$) and then the region of exponential convergence when frequency adaptation was used.
Our numerical experiments show that these two regions approximately coincide, even for small coupling ($K<10$). This suggest that frequency adaptation enters the exponential regime inside the synchronization region of the normal phase oscillator, even for small $K$.

In \cref{fig:sin1_entrain} we show an example of such convergence where convergence becomes exponential when entering the synchronization region.
\cref{fig:entrain_test} shows a superposition of the regions of exponential convergence
and the synchronization regions for a more complex periodic input. This suggests that our observation extends to more complex synchronization regions structures.
We numerically evaluated the behavior of the adaptive frequency phase oscillator for different types of input, different values of coupling and initial conditions for $\omega$.
For each experiment, we evaluated the synchronization regions of the oscillator for a given input without frequency adaptation together with the convergence behavior of $\omega$ with adaptation. 

\begin{figure}
\centering
%\subfigure[Example of convergence for small coupling]{
\includegraphics[width=.49\textwidth]{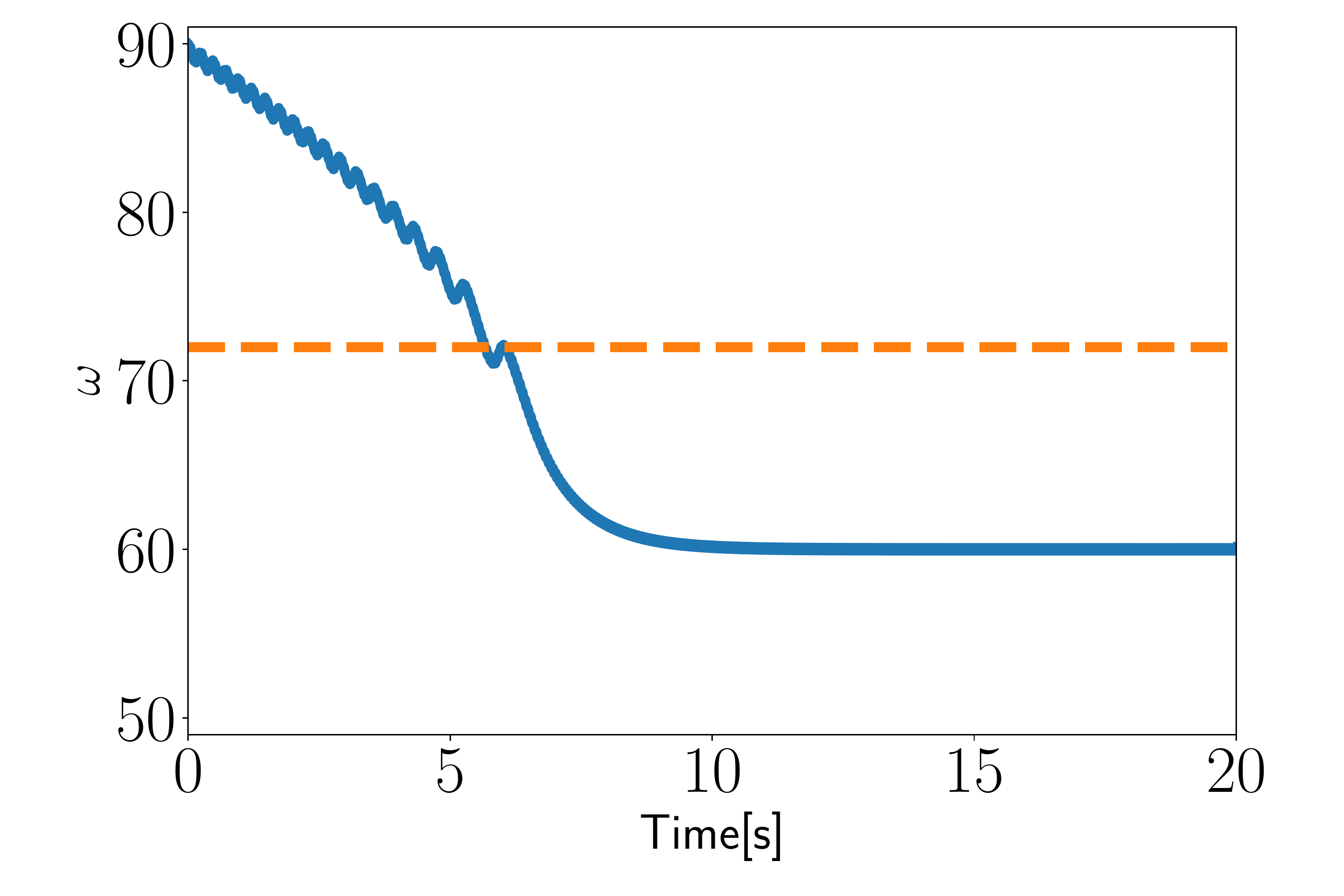}%}\label{fig:sin1_entrain1}
\caption{Example of convergence of $\omega$ for small coupling ($K=20$). The input signal is $F(t) = \cos(60 t)$, $\omega(0) = 90$. The vertical dashed line shows the limit of the synchronization region, we notice that convergence becomes exponential when the frequency of the oscillator enters in it.
}\label{fig:sin1_entrain}
\end{figure}

\begin{figure}
\centering
  \includegraphics[width=0.45\textwidth]{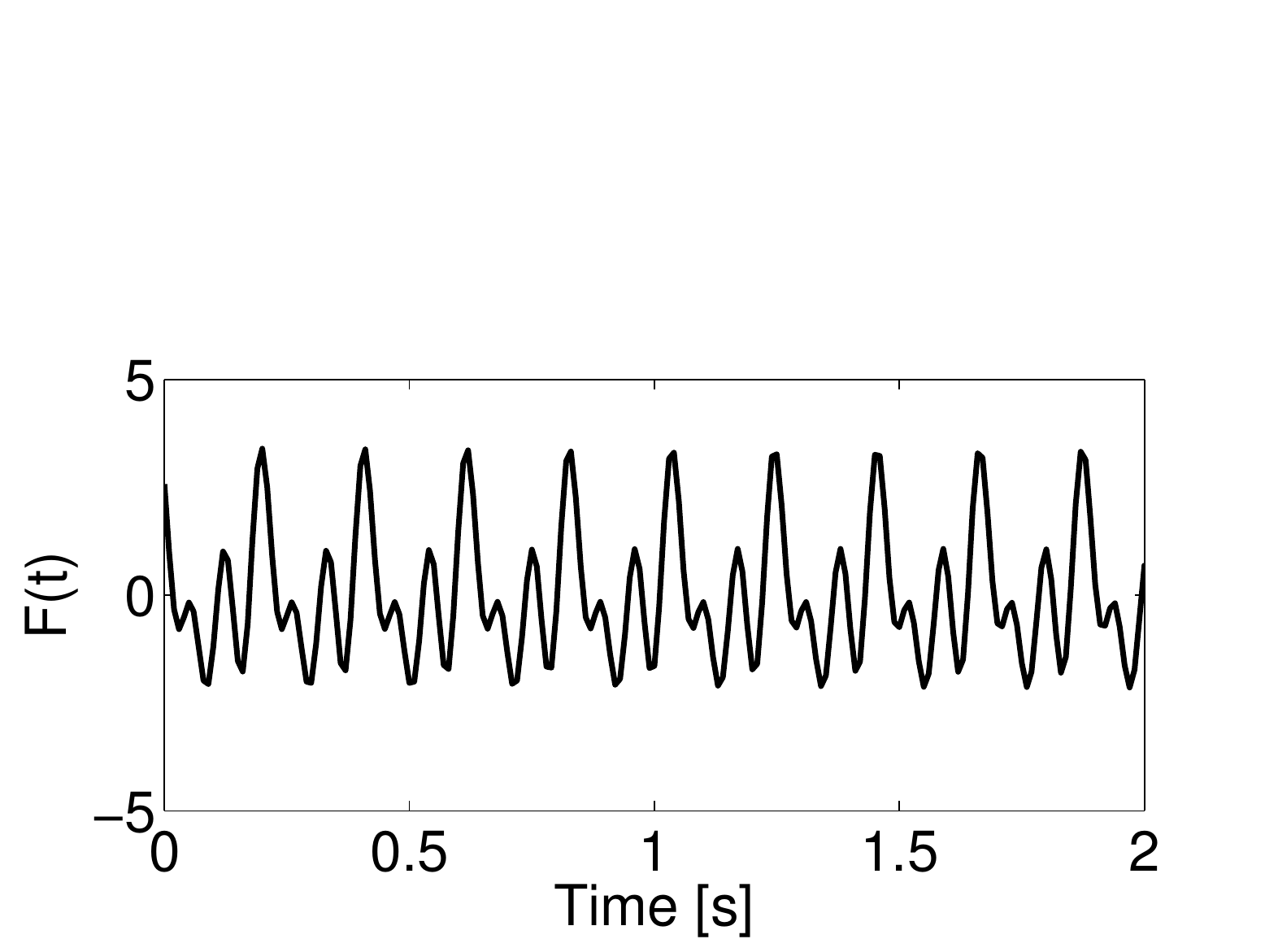}  \includegraphics[width=.46\textwidth]{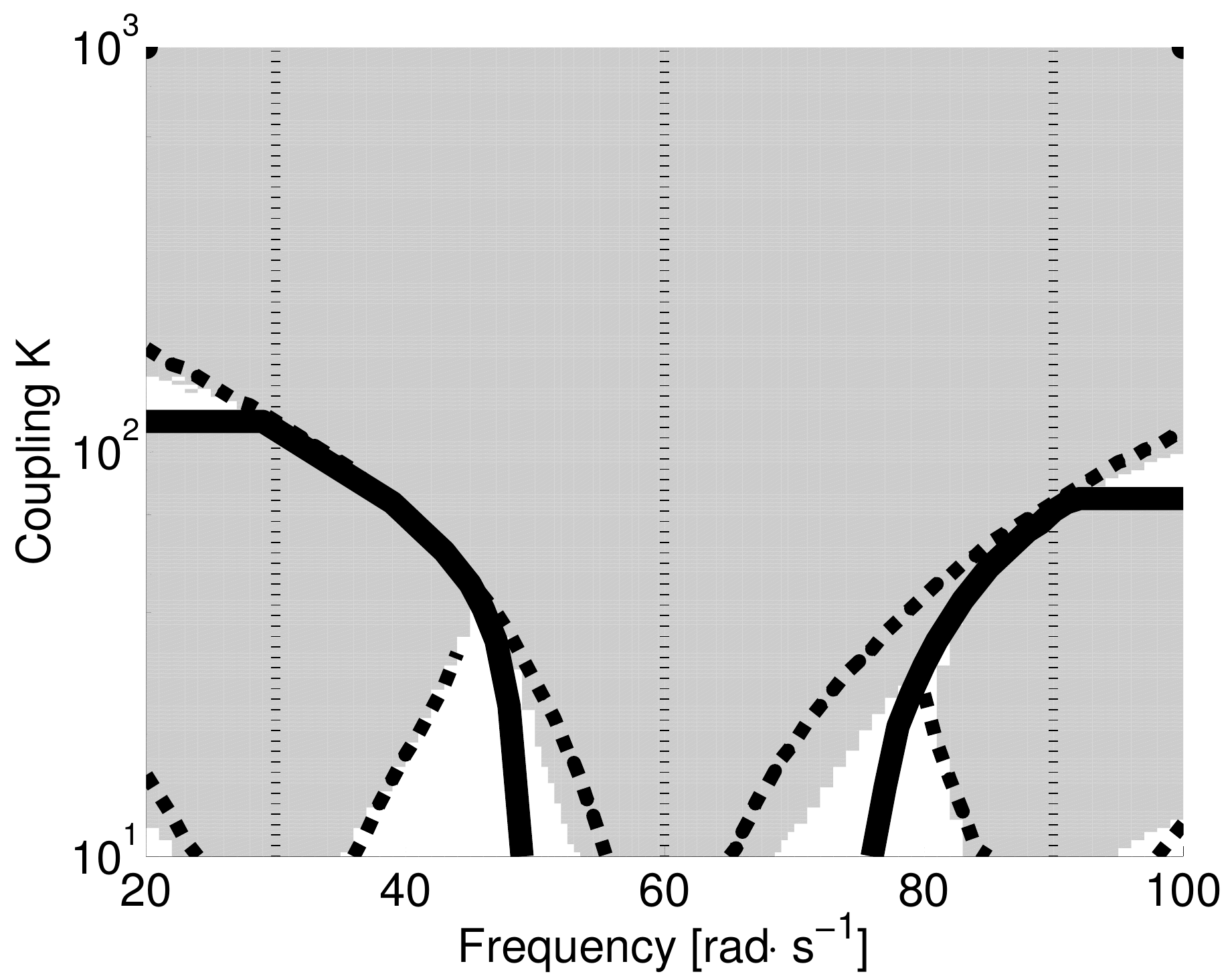}
  \caption{
The left graph show the input signal, $F(t) = 1.3\cos(30t + 0.4) + \cos(60t) + 1.4\cos(90t +1.3)$. The right graph shows: 1) the three synchronization regions of the phase oscillator without frequency adaptation associated to the three frequencies of the input (delimited by the dashed lines) and 2) the regions where where exponential convergence is numerically detected for the adaptive frequency phase oscillator (grey areas). We notice that both synchronization regions and regions of exponential convergence approximately overlap. For moderate coupling $K$, $\omega$ converges to one of the three frequency components of the input depending on its initial condition. The thick black line separate these three regions (e.g. the bottom left region shows the initial values of $\omega$ that converge toward 30, the center region towards 60 and the bottom right region towards 90). The vertical dotted lines represent the frequency components of the input.}\label{fig:entrain_test}
\end{figure}

The regions of exponential convergence match well the synchronization regions.
It must be noted that the numerical delimitation of the region of exponential convergence is not very precise, due to the complex interaction between the oscillator and the several frequency components of the input signal. This has to be taken into account to explain why the regions of exponential convergence slightly exceeds the regions of entrainment.
This observation enables to bridge the results we previously derived for small coupling $K<<1$ \cite{righetti06b} where we showed that convergence depends on the frequency components and their associated amplitude \eqref{eq:small_coupling} and results we derive in this article for strong coupling.
Albeit adaptation of frequency is different from mere synchronization, our numerical results suggest that the structure of the synchronization regions is critical in the convergence of the adapted frequency. 

\subsection{Extracting frequencies from a chaotic signal}
Thus far we only treated cases where the input was periodic.
Here we show the behavior of the adaptive frequency oscillator for inputs that are not periodic but possesses a localized peak in their frequency spectrum. The oscillator can extract this frequency from the signal, a non trivial task.
Further, we show that the maps derived in the previous section accurately predict the behavior of the system.
We consider the Lorentz system in its standard chaotic regime
\begin{align}
    \dot{x} & = 10 (y - x)\\
    \dot{y} &= x ( 28 - z) - y\\
    \dot{z} &= xy - \frac{8}{3}z
\end{align}
and use the state variable $z(t)$ as an input to the adaptive frequency oscillator. We center the variable to ensure that zero crossings happen, i.e. we use $F(t) = z(t) - 23$ as an input to the adaptive frequency oscillator. A Fourier transform of the signal shows a clear peak at frequency $8.42$ $\textrm{rad}\cdot\textrm{s}^{-1}$. Figure \ref{fig:lorentz} shows the result of the numerical simulation. The oscillator adapts $\omega$ to the correct frequency component, i.e. it is capable of extracting the major frequency present in the chaotic signal. 
Further, the computed maps accurately predict the behavior of the system.
Note here that the extraction of this frequency is not trivial as it
would require performing a FFT or a similar operation.

\begin{figure}
    \centering
    \includegraphics[width=0.49\textwidth]{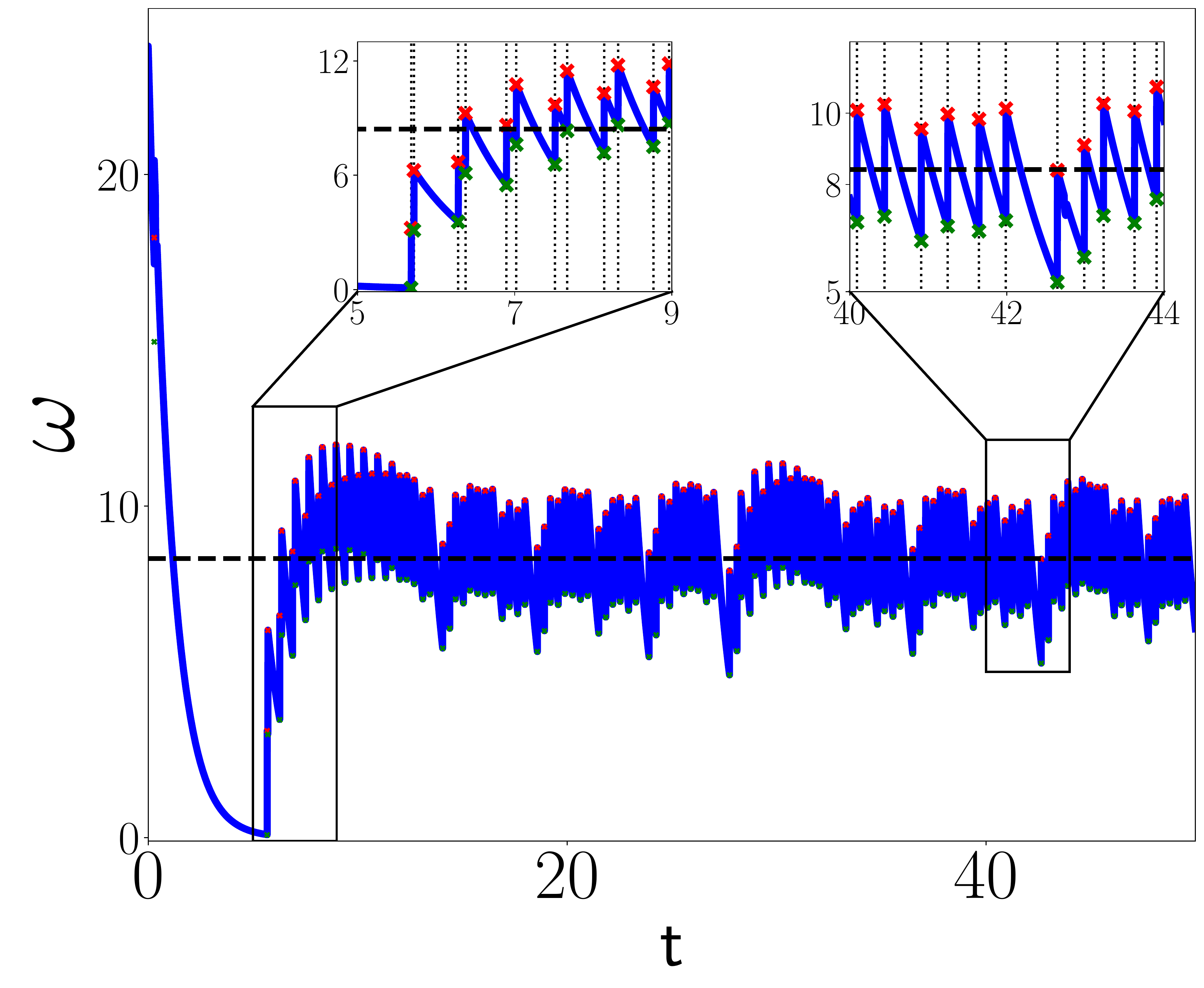}
    \includegraphics[width=0.49\textwidth]{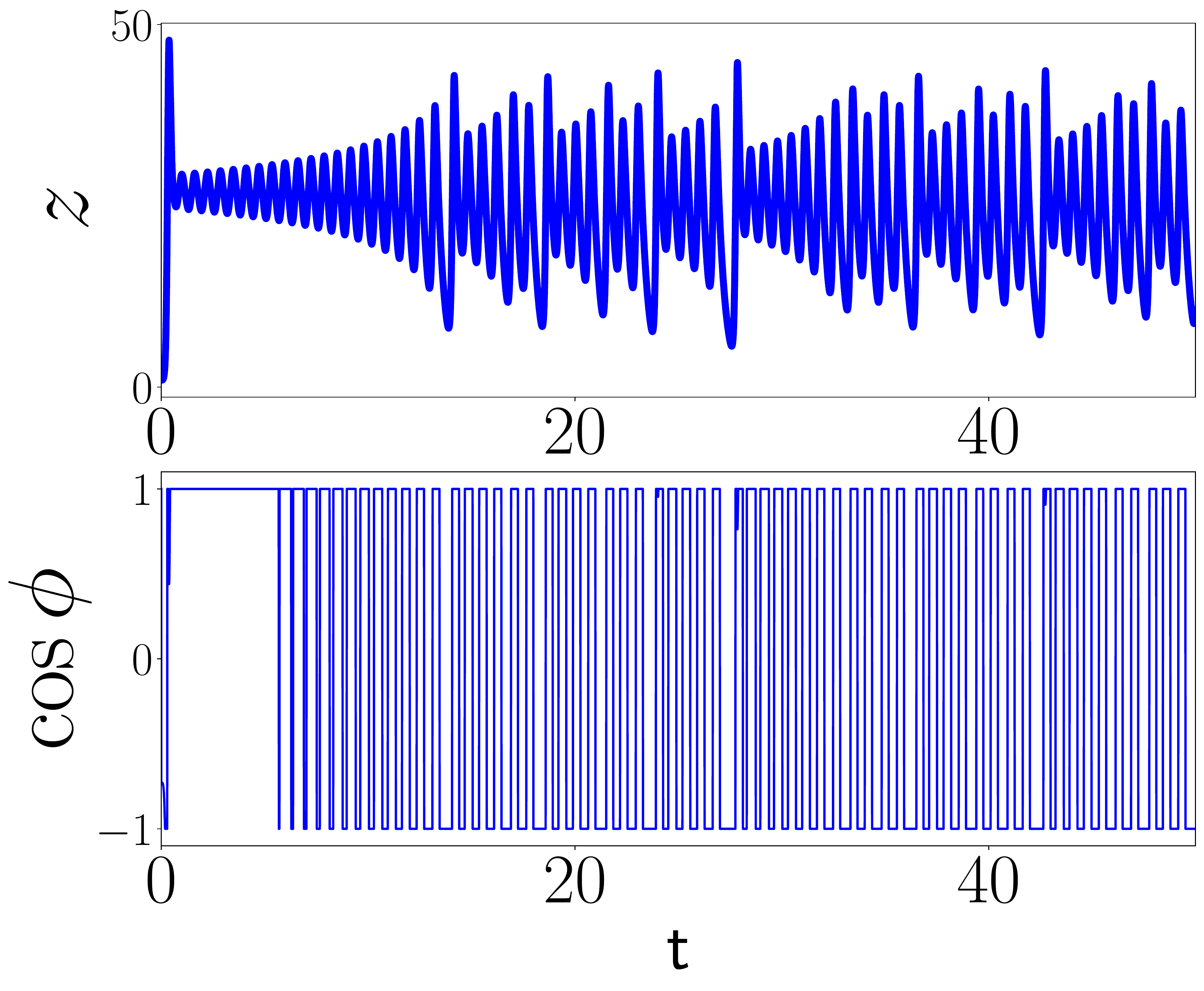}
    \caption{Numerical experiments with Lorentz system. The left graph shows the time evolution of $\omega$ when the input signal is the centered $z$ variable of the Lorentz system ($F(t) = z(t) - 23$), the dashed black line shows the frequency peak of the input. The crosses show the predictions from the $\omega^+$ (red cross) and $\omega^-$ (green cross) maps. In the zoomed plots, the vertical dashed lines correspond to the zero-crossings of the input. The right graph shows $z(t)$ (top) and the "output" of the oscillator $\cos\phi$.}
    \label{fig:lorentz}
\end{figure}

\subsection{Non-periodic inputs with discrete spectra}\label{sec:aperioc}
The behavior of the adaptive frequency oscillator when
the input $F(t)$ has a discrete spectra but is not periodic is well defined
in the small coupling case: the oscillator frequency converges
to one of the frequency component present in the input spectrum \cite{righetti06b}. In fact, we observe similar behavior concerning exponential convergence in synchronization regions than in the periodic case discussed previously.
However, we empirically found that after $K$ reaches a certain value,
the oscillator's frequency does not converge anymore.
However, the discrete maps defined in
 \cref{eq:discrete_maps_ti1,eq:discrete_maps_ti2} are still capable to accurately
predict the system behavior after each slow-fast events. As an example, \cref{fig:conv_predict_aperiodic3} shows the evolution of $\omega$ compared to
the prediction of each slow-fast event for such a case. We can see that the discrete maps are able to very well predict the non-trivial, non-periodic behavior of the system. This empirically supports the validity of our analysis for complex inputs.
Our observation also implies that for non-periodic signals, convergence
depends on coupling strength, which is not the case for periodic signals, where the system always converges to a multiple of the signal frequency.
\begin{figure}
\centering
\includegraphics[width=0.75\textwidth, trim=50 0 50 0, clip]{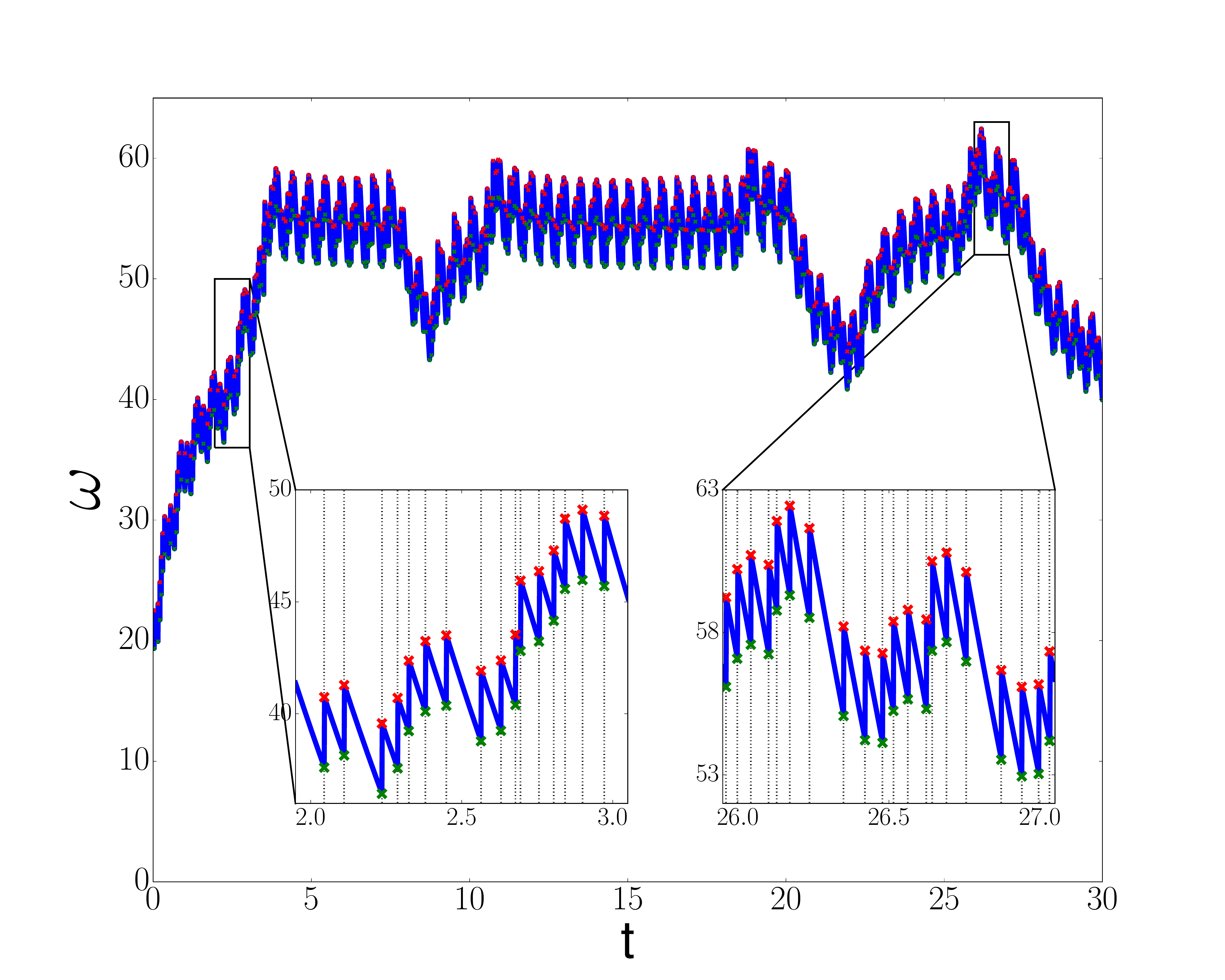}
\caption{Time evolution of the adaptive frequency oscillator for $F(t) =  1.3\cos(30t) + \cos(30\sqrt{2}t) + 1.4\cos(\frac{30\pi}{\sqrt{2}}t)$, where $\lambda =1$, $K=10^7$, $\omega_F=30$. The graph shows the evolution of $\omega(t)$ (blue line), the predictions $\omega_n^+$ (red cross) an $\omega_n^-$ (green cross) taken from \cref{eq:discrete_maps_ti1,eq:discrete_maps_ti2}. In the zoomed plots, the vertical dashed lines correspond to $F(t) = 0$.}\label{fig:conv_predict_aperiodic3}
\end{figure}

\subsection{Tracking changing frequencies}
An adaptive frequency oscillator will also be able to track a time-varying frequency.
We have seen earlier that the frequency adaptation is exponential \eqref{eq:exp_convergence} with convergence rate $\lambda$. 
Therefore, while the dynamics of $\omega$ is nonlinear,
its average behavior resembles that of a first-order linear low pass filter with 
cutoff frequency $\lambda$ $\mathrm{rad}\cdot\mathrm{s}^{-1}$. We assume here that
the low-pass filter acts directly on the frequency of the input signal $F(t)$, which the system does not have explicit access to.
In this case, changing frequencies will be tracked properly only when $\dot{\omega}_F < \lambda$.
To support this claim, we numerically evaluate the response of $\omega(t)$ when the input's frequency changes periodically (i.e. we aim to see how fast can $\omega$ adapt to an input with a time-varying frequency).
We use the input signal $F(t) = \sin(\psi)$, with $\psi = \omega_F t + \frac{1}{\omega_C} \sin(\omega_C t)$
so the instantaneous frequency of the input is $\dot{\psi} = \omega_F + \cos(\omega_C t)$, i.e. the frequency of the input is oscillating around $\omega_F$ at frequency $\omega_C$.
The frequency response analysis aims to study the ratio and phase differences between the time series $\omega(t)$ and $\omega_F + \cos(\omega_C t)$. This enable us to study the first order linear response of the frequency adaptation mechanism as shown in \cref{fig:freqResp}. We notice in the figure all the characteristics
of a first order low pass filter: amplitude reduction of -3dB and $\frac{\pi}{4}$ phase shift at the cutoff frequency, amplitude reduction of 20dB per decade.

\begin{figure}
\includegraphics[width=.49\textwidth, clip=true, trim=0 0 0 0]{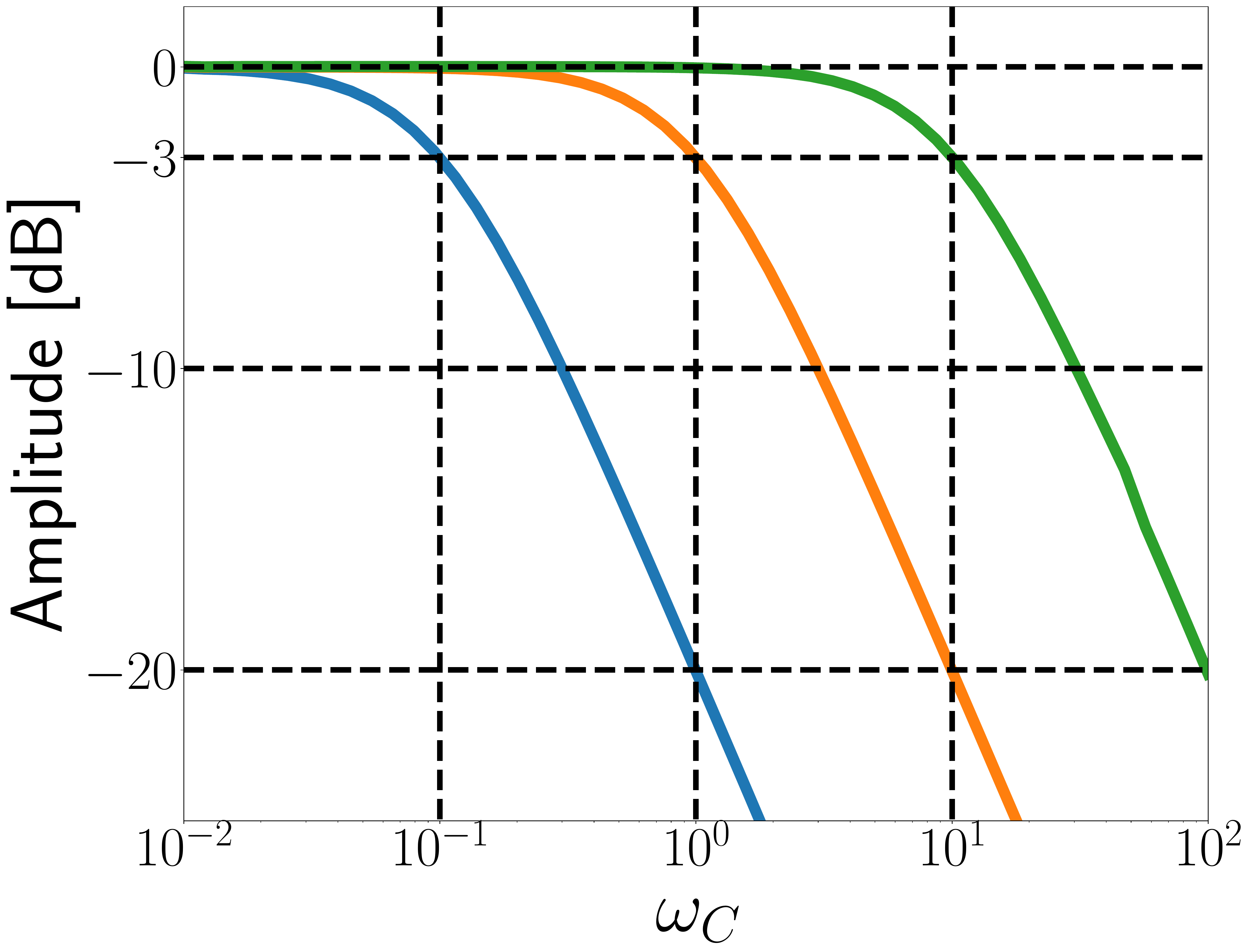}
\includegraphics[width=.49\textwidth, clip=true, trim=0 0 0 0]{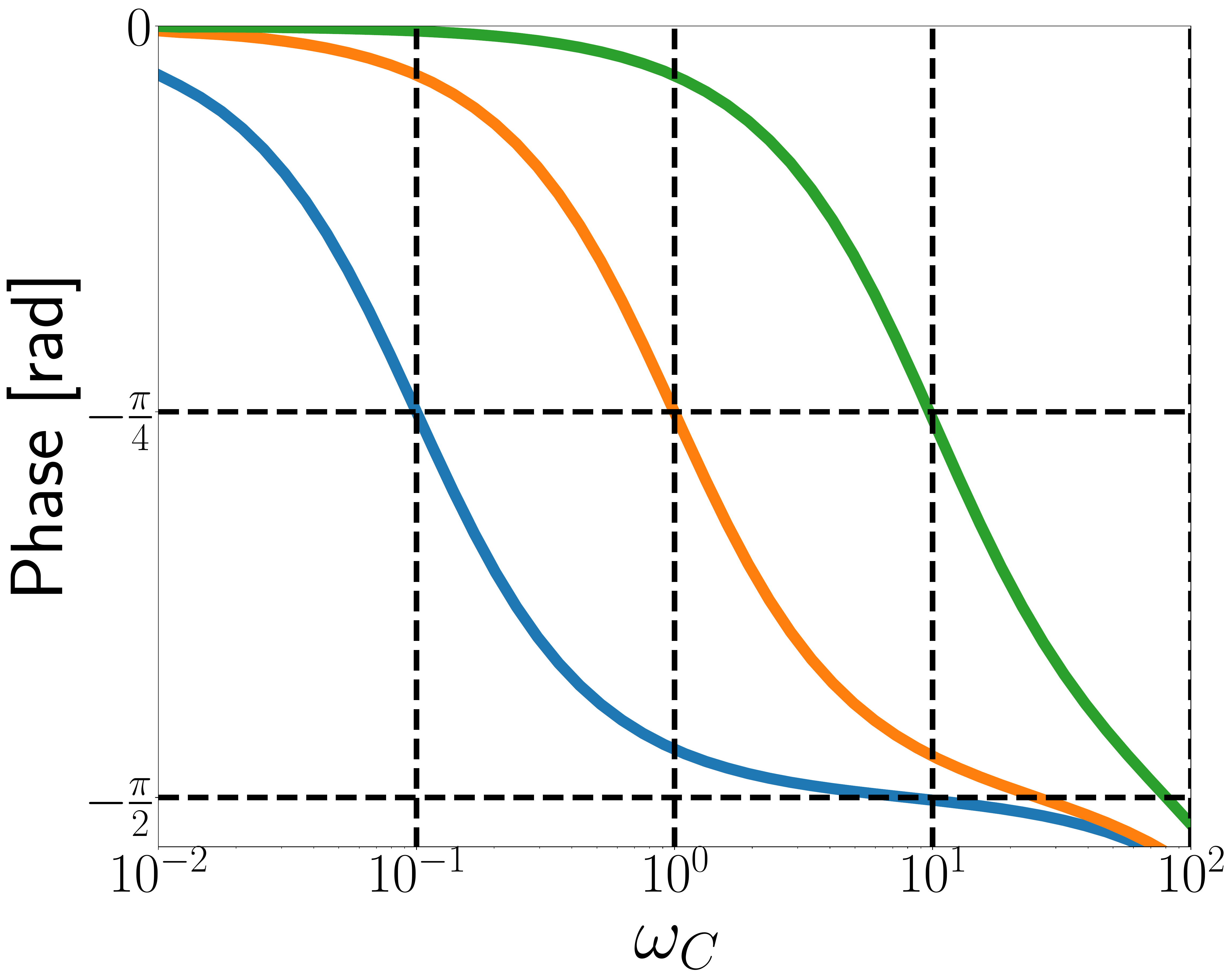}
\caption{Frequency response experiments. We show both the magnitude (left) and the phase delay (right) of the response of $\omega(t)$ for $\lambda = 1$ (orange), $\lambda=10$ (green) and $\lambda=0.1$ (blue) when the adaptive oscillator is subjected to an input with time-varying frequency. The time varying frequency is $\dot{\psi} = \omega_F + \cos(\omega_C t)$. The frequency response is computed using the ratio of the Hilbert transforms of $\dot{\psi}$ and $\lambda\omega(t)$ in steady-state (i.e. after the transient dynamics of $\omega(t)$).
}\label{fig:freqResp}
\end{figure}

These results can also explain the frequency tracking limitations empirically observed in \cite{buchli08} in the case of a pool of oscillators coupled with a negative feedback loop.
It is worth noting that while $\lambda$ can be chosen arbitrarily and will
lead to any desired convergence time to the input frequency, the resulting
oscillation around the desired frequency $\omega_F$ will have amplitude $\Delta\omega = \lambda \pi$. Therefore, the resolution of frequency extraction is limited by the speed at which this frequency is recovered. This is not surprising as it is reminiscent of fundamental time-frequency resolution results in signal processing.

%%%%%%%%%%%%%%%%%%%%%%%%%%%%%%%%%%%%%%%%%%%%%%%%%%%%%%%%%%%%%%%%%%%%%%%%%%%%

%%%%%%%%%%%%%%%%%%%%%%%%%%%%%%%%%%%%%%%%%%%%%%%%%%%%%%%%%%%%%%%%%%%%%%%%%%%%

%\subsection{Frequency resolution and time of convergence}

\section{Pool of adaptive frequency oscillators}\label{sec:pool}
In this section, we study the frequency adaptation mechanism
in a more complex system involving several oscillators coupled via negative feedback together with an adaptive amplitude mechanism.
Our previous work numerically investigated how a pool of adaptive frequency oscillators coupled via a negative mean field could do frequency analysis of signals, with discrete, continuous and time-varying spectra \cite{buchli08} but without any amplitude adaptation.
% One problem with the pool of oscillators is that the resolution of the analyzed frequency spectrum is highly dependent on the number of oscillators present in the pool.
Here, we study the method introduced in \cite{righetti05b} for robot control,
which additionally associates to each oscillator a variable encoding the amplitude of the corresponding oscillation.
We chose this system as this idea has been used in several robotics applications \cite{gams09,Petric:2011dr,righetti06}.
In these applications, a periodic pattern is learned from demonstrations using
the adaptive mechanism. The resulting stable oscillations generated by the system are then 
used as a controller (potentially adding coupling between oscillators after learning to ensure proper phase relations). Feedback can also be included after encoding the pattern to react to unexpected disturbances, for example to adapt walking patterns online \cite{righetti06}, thanks to the encoded stable limit cycle.
To our knowledge the method was never studied rigorously, in the weak or strong coupling regime, despite its use in several applications.

We show below the persistence of locally invariant slow manifolds similar to those identified for a single oscillator (\Cref{sec:2}) when the feedback loop and multiple oscillators are introduced.
We also show that amplitude adaptation is mainly due to the slow dynamics. Convergence is controlled through the disappearance of these slow manifolds as the amplitude gets properly adapted.
Interestingly, in the case of a single oscillator with the feedback loop, a new type of slow manifold appears which alters the frequency adaptation mechanism as the amplitude variable is adapted.
However, hyperbolicity is not preserved when introducing multiple oscillators and it is not clear if slow invariant manifolds of this type still exist. Numerical experiments show that a network of such oscillators can reproduce complex input signals. 
While for medium coupling we observe amplitude and frequency convergence, numerical simulations suggest that with strong coupling, convergence is not guaranteed anymore although the system properly reproduces the input signal. This numerical result suggest
certain limitations when using networks of adaptive frequency oscillators with strong coupling.

\subsection{System description}
The system consists of a set of adaptive frequency oscillators
coupled together via a negative feedback loop. To each oscillator,
we associate a new state variable $\alpha_i$ encoding the oscillator output amplitude. The output of the system is the sum of the outputs of the oscillators. The system is described as
\begin{align}
\dot{\phi}_i &= \lambda\omega_i - K F(t) \sin\phi_i\\
\dot{\omega}_i &= - K F(t) \sin\phi_i\\
\dot{\alpha}_i &= \eta F(t) \cos\phi_i\\
F(t) &= I(t) - \sum_{i=1}^{N}\alpha_i\cos\phi_i
\end{align}
where $\eta>0$ and we assume $\alpha_i=0$ at $t=0$. We call output of the system the sum $\sum_{i=0}^{N}\alpha_i\cos\phi_i$. $I(t)$ is an arbitrary input signal (typically a periodic signal) assumed to be $C^\infty$.
The intuition behind the behavior of the system is as follows: each oscillator will
adapt its frequency to one frequency component of the input, and then adapt its amplitude $\alpha_i$ until this frequency component disappears from the error signal $F(t)$. The remaining oscillators can then adapt their frequency to the remaining frequency components until $F(t)=0$ and the system is able to reproduce the input completely.

\begin{remark}
If the input is periodic $I(t) = \sum_{i=1}^N A_i \cos(\Gamma_i t + \gamma_i)$ (where $A_i$, $\Gamma_i$ and $\gamma_i$ are real constants), for a system with at least $N$ oscillators, setting $\lambda \omega_i = \Gamma_i$, $\alpha_i = A_i$, $\phi_i = \Gamma_i t + \gamma_i$ for the first N oscillators and $\alpha_i = 0$ for the remaining ones is a solution of the system. The system's output perfectly reconstructs the periodic input signal and $F(t) = 0$ for all $t$.
\end{remark}

\subsection{Singular orbits for a single oscillator}
First we study the singular orbits of the slow-fast system with a single oscillator to understand the effect of feedback and amplitude adaptation.
Again, we use the change of variable $\Omega_i = \phi_i -\omega_i$  and set $\epsilon = \frac{1}{K}$ to study the singularly perturbed system
\begin{align}
\epsilon\dot{\omega} &= - \sin(\Omega + \omega) \left( I(\theta) - \alpha\cos(\Omega + \omega) \right)\label{eq:pool_one_oscill1}\\
\dot{\Omega} &= \lambda\omega\label{eq:pool_one_oscill2}\\
\dot{\alpha} &= \eta \cos(\Omega + \omega) \left( I(\theta) - \alpha\cos(\Omega + \omega) \right)\label{eq:pool_one_oscill3}\\
\dot{\theta} &= 1\label{eq:pool_one_oscill4}
\end{align}

\begin{theorem}
For $\epsilon$ sufficiently small, there exist infinitely many slow invariant manifolds for the flow of \cref{eq:pool_one_oscill1,eq:pool_one_oscill2,eq:pool_one_oscill3,eq:pool_one_oscill4}. They consist of simply connected, compact subsets of $\mathbb{R}^4$ with one of the following three forms
\begin{equation}\begin{array}{lll}
\Mminus:\ \ & \omega = (k\pi - \Omega)(1 + \frac{\epsilon\lambda}{(-1)^k I(\theta) - \alpha}) + O(\epsilon^2),&\quad \alpha < (-1)^k I(\theta),\ k\in\mathbb{Z} \nonumber\\
\Mplus:\ \ & \omega = (k\pi - \Omega)(1 + \frac{\epsilon\lambda}{(-1)^k I(\theta) - \alpha}) + O(\epsilon^2),&\quad \alpha > (-1)^k I(\theta),\ k\in\mathbb{Z} \nonumber\\
\MF:\ \ & \alpha \cos(\Omega + \omega) = I(\theta) + O(\epsilon),&\quad \sin(\Omega + \omega) \neq 0,\ \alpha \neq 0 \nonumber
\end{array}\end{equation}
Manifolds of the type $\Mminus$ are locally attracting and ones of the type $\Mplus$ are locally repelling and the ones of the form $\MF$ are locally attracting if $\alpha > 0$ and repelling otherwise.
\end{theorem}

\begin{proof}
The proof for $\Mminus$ and $\Mplus$ follows the same reasoning than for the proof
of Theorem \ref{th:slow_manifolds} and we omit it for brevity. For $\MF$, we verify that $\alpha \cos(\Omega + \omega) = I(\theta)$ is a solution to \cref{eq:pool_one_oscill1} when $\epsilon=0$ so simply connected compact subsets satisfying this relation are candidate manifolds.
The Jacobian of the fast system on this critical manifold when $\epsilon=0$ has only one non-zero eigenvalue $-\alpha \sin^2(\Omega + \omega)$ with eigenvector $[1,0,0,0]^T$ so
the direction transverse to the critical manifolds has a non zero eigenvalue as long
as $\sin(\Omega + \omega) \neq 0$ and $\alpha \neq 0$ and these critical manifolds are hyperbolic. Fenichel's theorem can then be invoked to prove the existence of locally invariant slow manifolds $O(\epsilon)$ close to the critical ones. These manifolds are attracting when the sign of the non-zero eigenvalue is negative (and repelling otherwise) which is defined by the sign of $\alpha$.
\end{proof}
We note $\MminusZ$, $\MplusZ$ and $\MFZ$ the critical manifolds on which we study the singular flow.

\subsubsection{Critical orbits on $\MminusZ$ and $\MplusZ$}
Interestingly, the system with the feedback loop and amplitude adaptation still has invariant manifolds $\Mminus$ and $\Mplus$ similar to the ones seen in \Cref{sec:2} with similar $O(\epsilon)$ shapes. However, the feedback loop and amplitude adaptation change the conditions for exiting the manifolds. Indeed, the input zero-crossings does not anymore trigger the fast events. Instead, the fast event is triggered when $\alpha + (-1)^{k+1}I(\theta)$ changes sign. This means that as $\alpha$ increases, the duration of the flow on or close to $\MminusZ$ decreases. When $\alpha >(-1)^{k}I(\theta)$ the flow is not anymore in proximity of attracting slow invariant manifolds of type $\Mminus$.
As for the previous case, the critical orbits on $\MplusZ$ and $\MminusZ$ remain the same for $\omega$ and $\Omega$, i.e.
\begin{align}
    \omega &= -\Omega(0) \mathrm{e}^{-\lambda t} + O(\epsilon)\\
    \Omega &= k\pi +\Omega(0) \mathrm{e}^{-\lambda t} + O(\epsilon)
\end{align}
Furthermore, on these critical manifolds (i.e. when $\Omega + \omega = k\pi$) we have
\begin{align}
    \dot{\alpha}     &= \eta \left( (-1)^{k} I(\theta) - \alpha \right)
\end{align}
and so $\dot{\alpha}>0$ on $\MminusZ$ and $\dot{\alpha}<0$ on $\MplusZ$.
On the attracting manifold, $\alpha$ increases but cannot go beyond the magnitude of the input $I(t)$ to remain on the manifold. This analysis shows that the critical orbit is such that
the amplitude adaptation increases at most to the maximum amplitude of $I(\theta)$. One can easily see that if the input is a simple cosine of amplitude $A$, then $\alpha$ will converge to $A$.
Since both frequency and amplitude adaptation are happening concurrently, it might be desirable to choose $\eta$ small enough with respect to $\lambda$ to ensure frequency convergence prior to amplitude convergence.

\subsubsection{Critical orbits on $\MFZ$}
Interestingly, there are slow locally invariant manifolds $\MF$ appearing when $\alpha \neq 0$. They are due to the feedback loop. On the critical manifolds $\MFZ$, the input $I(t)$ and the output $\alpha \cos(\Omega + \omega)$ are equal.
Writing $\omega = \pm\arccos(\frac{I(\theta)}{\alpha}) - \Omega + n2\pi$
with $n\in\mathbb{Z}$, the critical dynamics is of the form
\begin{align}
    \dot{\Omega} &= \lambda \left( \left(\pm\arccos(\frac{I(\theta)}{\alpha}) - n2\pi\right) - \Omega \right)\\
    \dot{\alpha} &= 0
\end{align}
We see that the amplitude $\alpha$ remains constant on $M_F$ and $\Omega$ will behave like a first-order low-pass filter on a signal of the form $\pm\arccos(\frac{I(\theta)}{\alpha}) - n2\pi$ with cutoff frequency $\lambda$.
Therefore $\Omega$ will follow this signal, i.e. it will increase or decrease until $\sin(\Omega+\omega)$ changes sign and the flow leaves
the manifold. This also implies that $\omega$ will tend to remain constant on $M_F$ if $\lambda$ when large enough compared to the rate of change of the input $I(\theta)$ (to ensure full magnitude of the frequency response).

Manifolds of type $\MFZ$ are bounded within the zeros of $\sin(\Omega+\omega)$ and the flow necessarily leaves them as $\Omega$ increases or decreases. As a simple example, consider $I(t) = A \cos(\omega_F t)$ and $\alpha = A$, then we have
$\dot{\Omega} = \lambda \left( \omega_F t - n2\pi - \Omega \right)$, and $\Omega$ follows the phase of the cosine input, which implies that $\omega$ remain near constant (for $\lambda$ large enough to track $\omega_F t$) since $\cos(\Omega + \omega) = \cos(\Omega_F t)$ on $\MFZ$. The dynamics exits $M_F$ as $\sin(\Omega + \omega)$ changes sign.

In summary, on the critical manifold $\MFZ$ we expect to qualitatively see no amplitude adaptation, an increase or decrease of $\Omega$ and close to no change in $\omega$.

\subsubsection{Critical fast orbits}
We now consider the critical fast dynamics
\begin{align}
    \dot{\omega} = -\sin(\Omega + \omega) \left(I(\theta) - \alpha \cos(\Omega+\omega) \right)
\end{align}
The fixed points of the dynamics correspond to points on the various critical manifolds and their stability properties correspond to the attracting/repelling nature of these manifolds.
The orbits of the fast critical dynamics consists of rapid transitions between neighborhoods of slow manifolds. When $\alpha=0$, we expect to see transitions between manifolds of type $\MminusZ$, similar to the dynamics studied in \Cref{sec:2}.
However, as $\alpha$ increases, we expect intermediary transitions to $\MFZ$, i.e. fast transitions will go from one $\MminusZ$ to $\MF$ and finally to another $\MminusZ$.
In that case, contrary to the system studied in \Cref{sec:2}, one fast event does not lead to an increase of $\pi$ for $\omega$ because the flow transits through a manifold of type $M_F$ "on the way" to the next $\MminusZ$.
\begin{figure}
    \centering
    \includegraphics[width=0.49\textwidth]{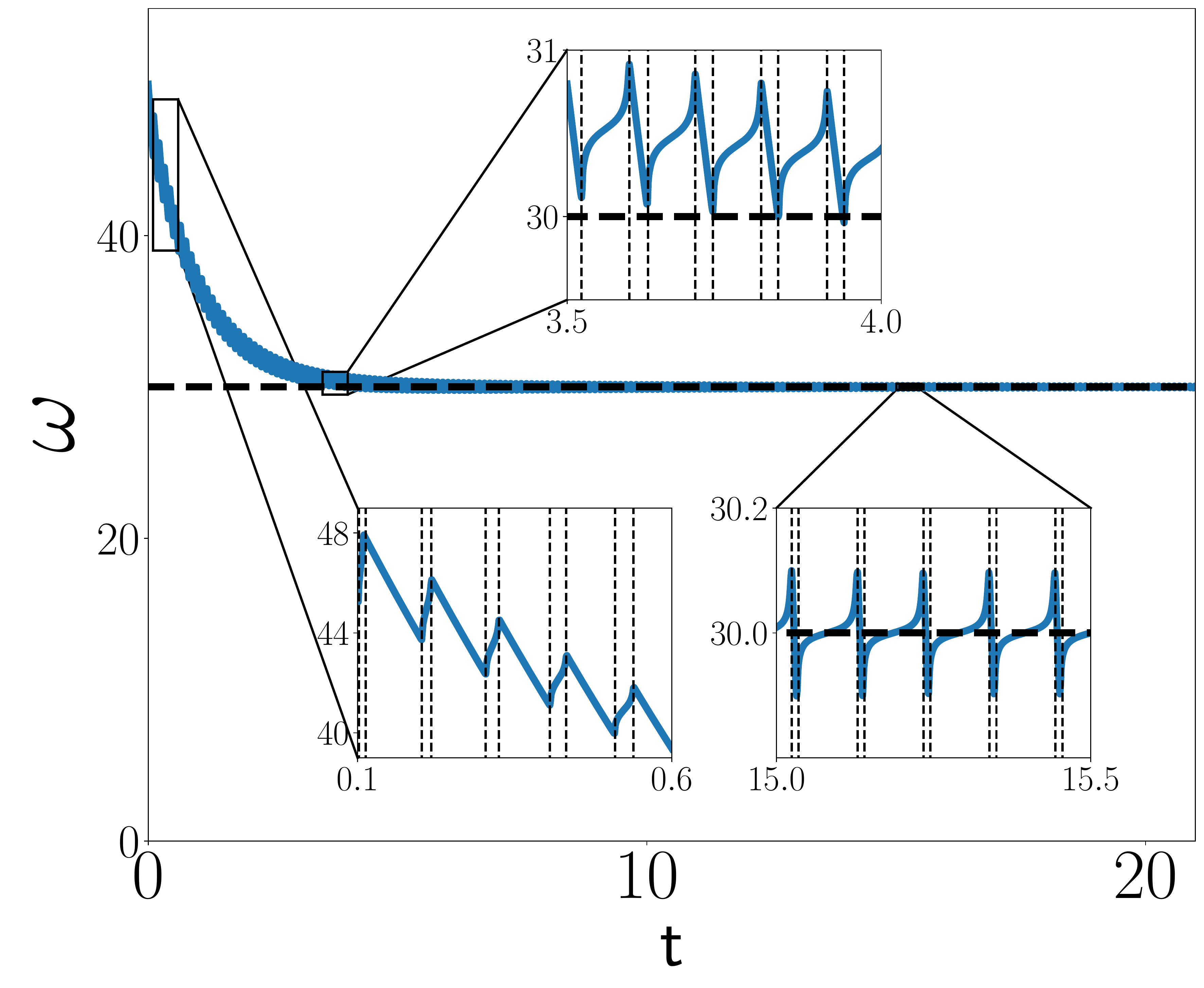}
    \includegraphics[width=0.49\textwidth]{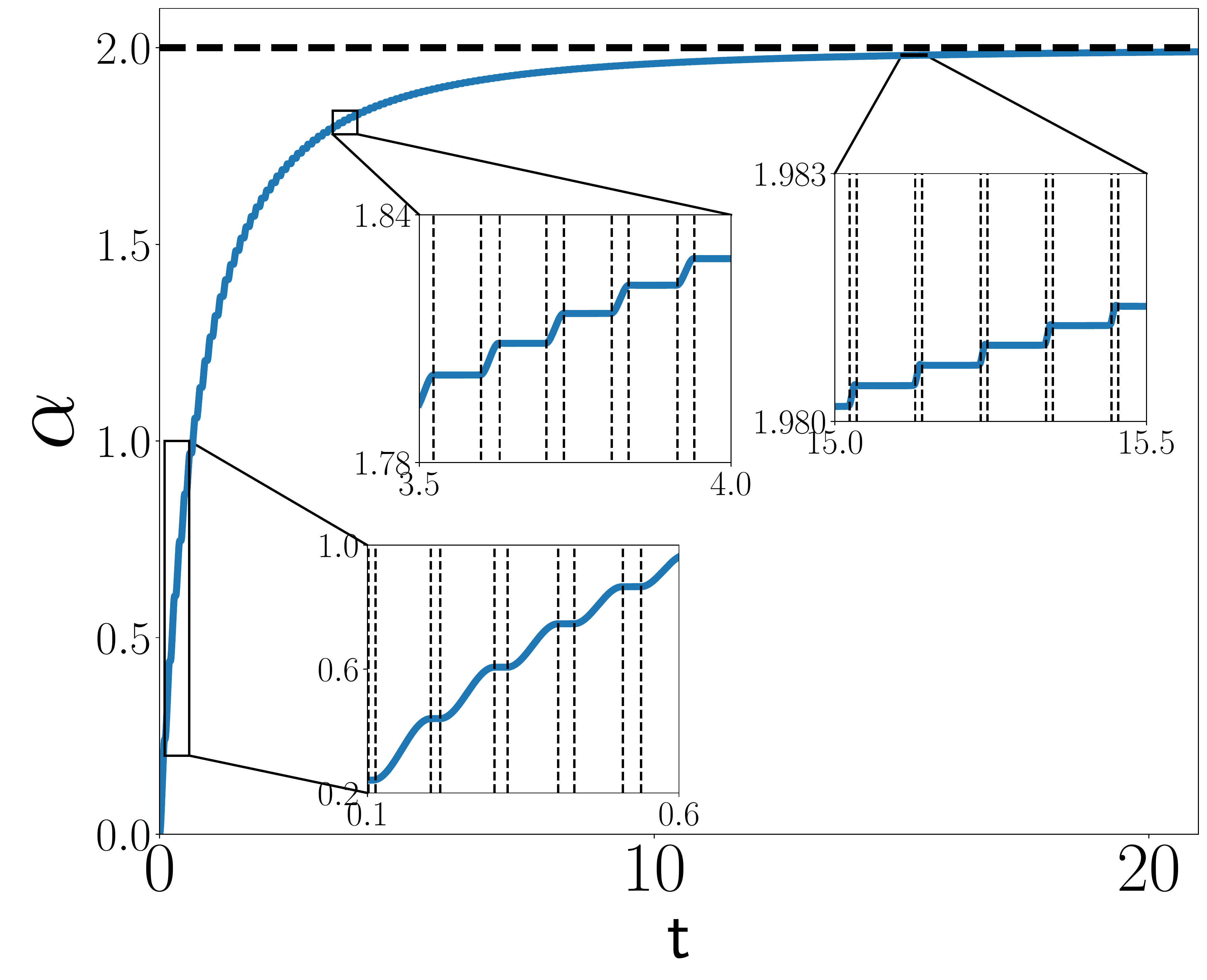}
    \includegraphics[width=0.49\textwidth]{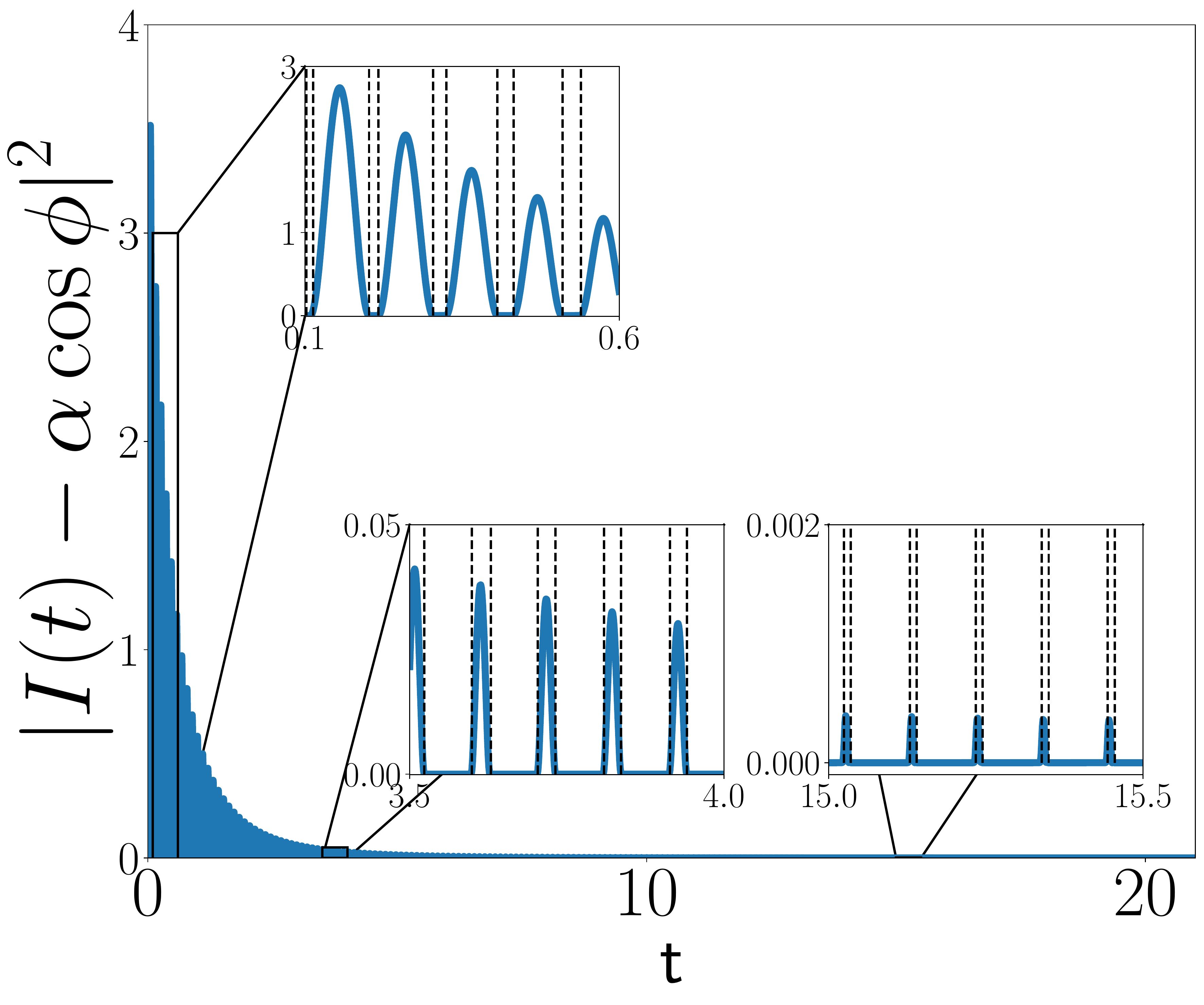}
    \caption{Example of frequency and amplitude adaptation for one oscillator ($K=10^5$, $\lambda=1$ and $\eta=2.$) and $I(t) = 2\cos(30 t)$. $\omega$ and $\alpha$ converge to $30$ and $2$ as expected. The vertical dashed lines enclose the regions where $|I(t) - \alpha\cos\pi|^2 < 10^{-5}$, (i.e. numerically close to 0).
    The qualitative behavior described in \ref{sec:qual_behav} is shown in the three quadrants of each plot: 1) at the beginning where the evolution of $\omega$ resembles that of the original adaptive frequency oscillator (\Cref{sec:2}), 2) in the middle of the convergence where increases of $\omega$ are reduced due to the presence of $\MF$ and 3) towards convergence where $\omega$ varies very little. Similarly we notice that $\alpha$ increases only when $\omega$ decreases (i.e. when the flow is close or on $\Mminus$).}
    \label{fig:pool_one_oscill}
\end{figure}
\subsubsection{Qualitative example of a critical behavior}\label{sec:qual_behav}
Taking all the critical orbits together we can describe the qualitative behavior of a complete critical orbit, for example when $I(t) = A \cos(\omega_F t)$, assuming that $\alpha(t=0) = 0$. At the beginning we expect transitions between stable manifolds of type $M_\pi^-$ with an increase of $\omega$ by $\pi$ during the fast event similar to the case studied in the previous sections. On $M_\pi^-$, the amplitude $\alpha$ will increase. As $\alpha$ increases, we will observe shorter slow events on $M_\pi^-$ and fast events where $\omega$ increases by a smaller amount than $\pi$ as the orbit reaches $M_F$. There, $\Omega$ will increase but $\omega$  and $\alpha$ will remain constant. Upon exiting $M_F$, $\omega$ will increase rapidly again until reaching  another $M_\pi^-$. Over a succession of two fast transitions and one slow event on $\MFZ$, i.e. $\MminusZ  \rightarrow \MFZ \rightarrow \MminusZ$, $\omega + \Omega$ will increase by $\pi$ but $\omega$ will have changed by less than $\pi$. Increasingly, most of this change will be attributed to $\Omega$ until the system converges to $\alpha = A$ and $\omega = \omega_F$. Interestingly, as $\alpha \to A$, the oscillations of $\omega$ will decrease and eventually disappear (i.e. the feedback loop enables perfect adaptation to $\omega_F$). A numerical illustration of this behavior is shown in Figure \ref{fig:pool_one_oscill} where we see all the features of the flow described in this section.

\subsection{Singular orbits for a pool of N oscillators}
We now extend our analysis to the case where there are $N$ oscillators coupled through the negative feedback loop.
The dynamics of each oscillator is
\begin{align}
\epsilon\dot{\omega}_i &= - \sin(\Omega_i + \omega_i) \left( I(\theta) - \sum_{j=1}^{N}\alpha_j\cos(\Omega_j + \omega_j) \right)\label{eq:pool_oscill1}\\
\dot{\Omega}_i &= \lambda\omega_i\label{eq:pool_oscill2}\\
\dot{\alpha}_i &= \eta \cos(\Omega_i + \omega_i) \left( I(\theta) - \sum_{j=1}^N\alpha_j\cos(\Omega_j + \omega_j) \right)\label{eq:pool_oscill3}\\
\dot{\theta} &= 1\label{eq:pool_oscill4}
\end{align}
\subsubsection{Existence of hyperbolic invariant slow manifolds}
The following theorem characterizes the slow invariant (hyperbolic) manifolds of the dynamics.
\begin{theorem}
For $\epsilon$ sufficiently small, there exist infinitely many slow invariant manifolds for the flow of \cref{eq:pool_oscill1,eq:pool_oscill2,eq:pool_oscill3,eq:pool_oscill4}. They consist of simply connected, compact subsets of $\mathbb{R}^{3N+1}$ 
of the following form
\begin{equation}\begin{array}{lc}
M_\pi:\ \  & \omega_i = (k_i\pi - \Omega_i)\left(1 + \frac{\epsilon\lambda (-1)^{k_i}}{I(\theta) - \sum_{j=1}^N(-1)^{k_j}\alpha_j)}\right) + O(\epsilon^2),\\ & \forall i \in \{1,\cdots,N\},\ k_i\in\mathbb{Z},\ I(\theta) \neq \sum_{j=1}^N (-1)^{k_j}\alpha_j \nonumber
\end{array}\end{equation}
Such invariant manifolds are attracting if $\alpha_i < (-1)^{k_i} I(\theta) - \sum_{i\neq j} \alpha_j (-1)^{k_j + k_i}$, $\forall i$, repelling if $\alpha_i > (-1)^{k_i} I(\theta) - \sum_{i\neq j} \alpha_j (-1)^{k_j + k_i}$, $\forall i$ and of saddle type otherwise.
\end{theorem}

\begin{proof}
The proof follows the same reasoning than the proof
of Theorem \ref{th:slow_manifolds}. We note here that the Jacobian of the critical fast dynamics (i.e. when $\epsilon=0$) has $N$ non-zero eigenvalues, all of the form $\alpha_i + (-1)^{k_i+1} I(\theta) + \sum_{i\neq j} \alpha_j (-1)^{k_j + k_i}$, their corresponding eigenvector has zero entries everywhere except for the row associated to $\omega_i$. It means that the unstable and stable manifolds associated to the critical manifold will be tangent to these eigendirections and the unstable and stable manifolds of $M_\pi$ will be $O(\epsilon)$ close to them.
\end{proof}

Interestingly, the slow manifolds remain in the case where $N$ oscillators are coupled together. Note however that the manifolds can now be of saddle type, i.e. they can now have stable and unstable transverse directions, leading to more complex dynamics. The associated unstable and stable manifolds are however aligned with the $\omega_i$ directions at $O(\epsilon)$, preserving the dynamics observed for a single oscillator. The flow on the critical manifolds are also similar, where $\Omega_i(t) = k_i\pi + \Omega_i(0) \mathrm{e}^{-\lambda t} + O(\epsilon)$ and  $\omega_i(t) = -\Omega_i(0) \mathrm{e}^{-\lambda t} + O(\epsilon)$. The dynamics
of $\alpha_i$ on one of these critical manifold is
\begin{align}
    \dot{\alpha}_i &= \eta \left(-\alpha_i + (-1)^{k_i} \left(I(\theta) - \sum_{j\neq i} (-1)^{k_j} \alpha_j \right) \right)
\end{align}
$\dot{\alpha}_i$ is positive when $\alpha_i < (-1)^{k_i} I(\theta) - \sum_{i\neq j} \alpha_j (-1)^{k_j + k_i}$, 
i.e. when the corresponding $\omega_i$ direction is attracting,  and negative otherwise. This is consistent with the findings of the single oscillator case. We note however that the amplitude adaptation now takes into account the amplitude contributions associated to the other oscillators (i.e. these contributions are removed from the input signal), potentially creating more complex interactions.
Qualitatively, the critical orbits described in the case of one oscillator with feedback are preserved when several oscillators are introduced.

\subsubsection{Non-hyperbolic candidate slow manifolds}
The most important difference here is that critical slow manifolds of type $I(\theta) = \sum_{j=1}^N \alpha_j \cos(\Omega_j + \omega_j)$ seen
in the one oscillator case are not normally hyperbolic.
Indeed, the eigenvalues of the critical fast dynamics Jacobian are all 0 except one which is equal to $-\sum_{i=1}^N \alpha_i \sin^2 (\Omega_i + \omega_i)$. The associated eigenvector is $[1, 0, 0, \cdots, \frac{\sin (\omega_i + \Omega_i)}{\sin (\omega_1 + \Omega_1)}, 0, \cdots, 0, \frac{\sin( \omega_j + \Omega_j)}{\sin (\omega_1 + \Omega_1)}, \cdots]$ where the non-zero entries are the rows associated to the $\omega_i$.
In this case, Fenichel theory cannot be applied to investigate the persistence of invariant slow manifolds at order $O(\epsilon)$ and it
is not directly obvious how one can study those. 
Indeed, the non-hyperbolic critical manifold is not nilpotent since one eigenvalue of the Jacobian is non-zero and techniques that handle non-hyperbolicity such as the blow-up method cannot be directly applied \cite{Jardon-Kojakhmetov_Kuehn_2019}.

The numerical experiments presented below suggest that solutions of the dynamics are indeed attracted towards $I(t) - \sum_i \alpha_i \cos\phi_i$. However, they also reveal that the dynamics of each oscillator can be rather complicated. 
Therefore, we conjecture that locally invariant manifolds of type $I(\theta) = \sum_{j=1}^N \alpha_j \cos(\Omega_j + \omega_j)$ do indeed exist and that the resulting dynamics includes complex interactions between oscillators. A formal analysis of this case remains however beyond the scope of this paper, and might be of limited interest for engineering applications due to the seeming lack of convergence to defined frequencies and amplitudes as discussed below.
\begin{figure}
\centering
\subfloat[$K=10$]{
  \includegraphics[width=.3\textwidth]{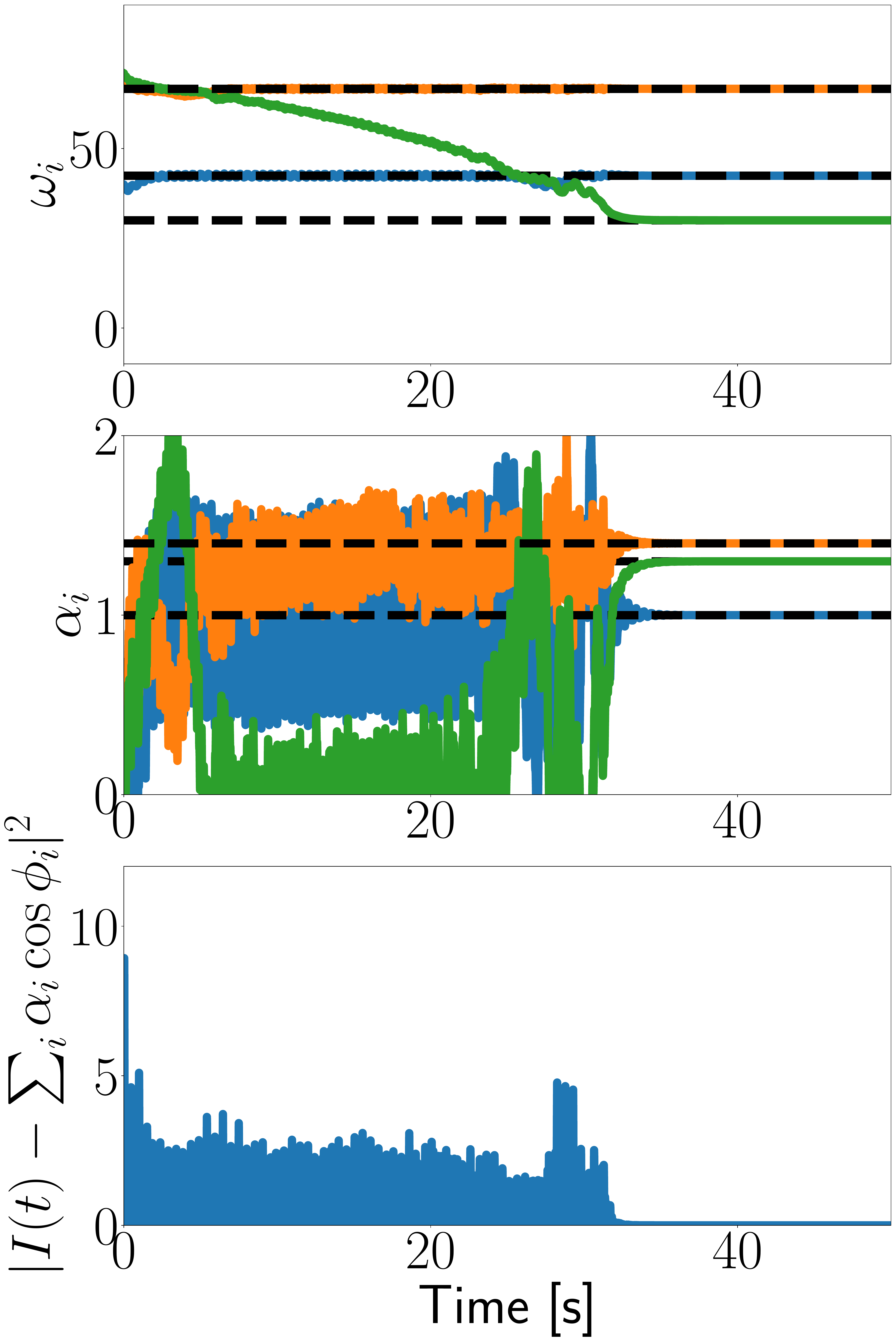}
}
\subfloat[$K=100$]{
  \includegraphics[width=.3\textwidth]{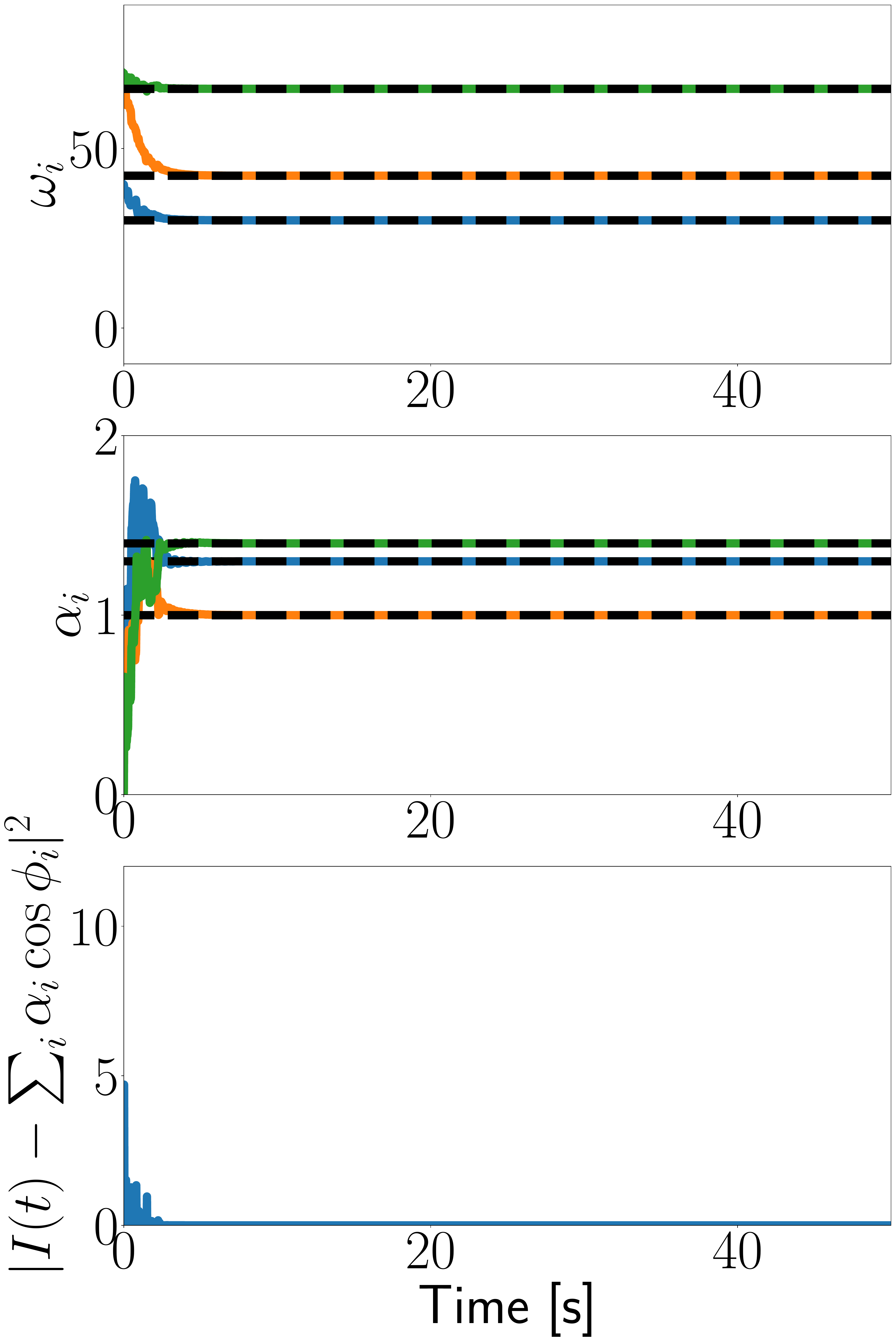}
}
\subfloat[$K=10000$]{
  \includegraphics[width=.3\textwidth]{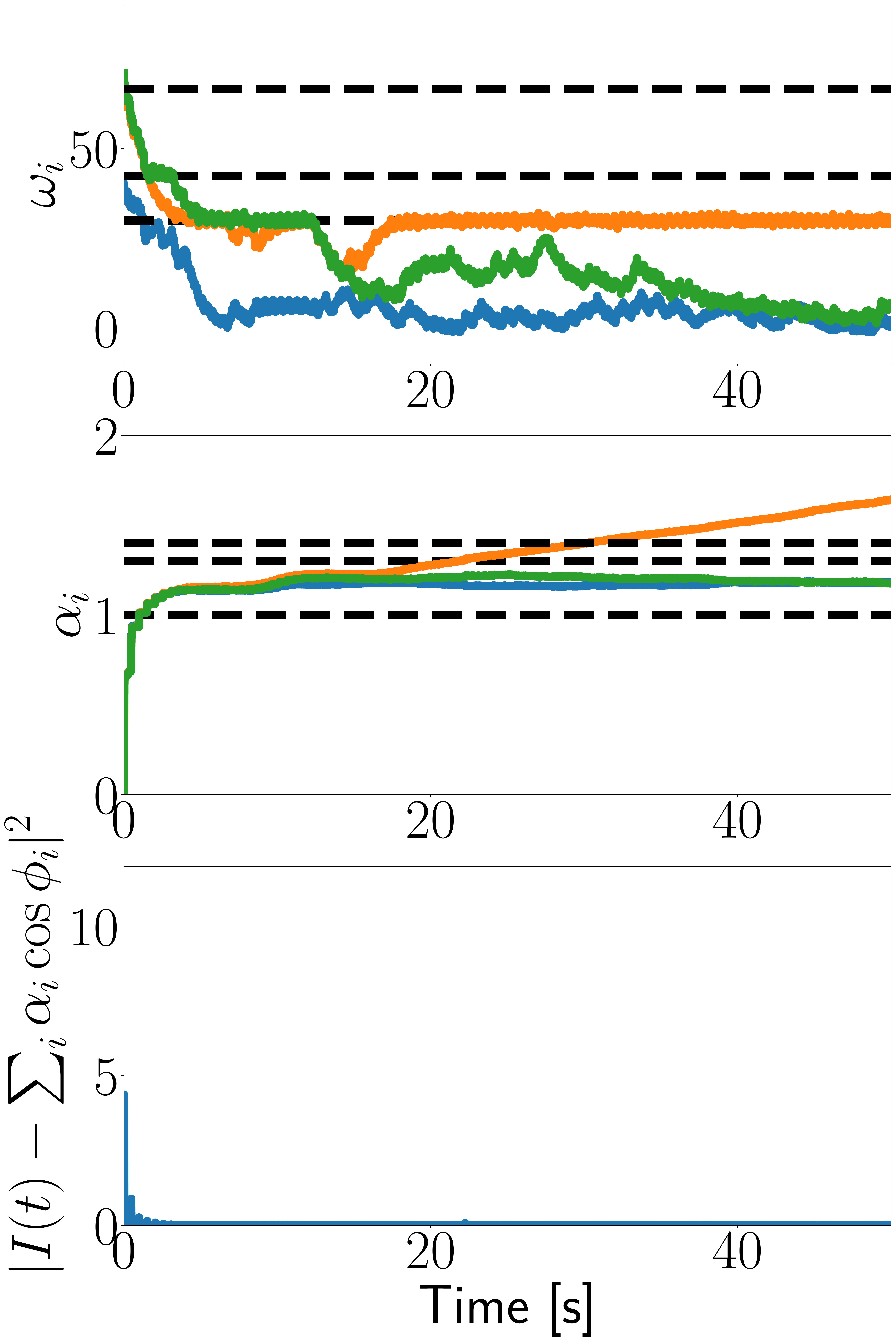}
}
  \caption{Examples of decomposition of the spectrum of an input signal $I(t) = 1.3\cos(30t) + \cos(30\sqrt{2}t) + 1.4\cos(\frac{30\pi}{\sqrt{2}}t)$ with a pool of $N=3$ oscillators with amplitude adaptation for three different coupling strengths. The parameters used in the simulations are $\lambda=1$ and $\eta=10$. For each experiment, the top graphs show the evolution of the state variables $\omega_i$ and $\alpha_i$, the bottom graph is the square error between the input signal $I(t)$ and the output of the pool of oscillators $\sum_{i=0}^N \alpha_i \cos \phi_i$. The dashed lines correspond to the input frequencies and amplitudes.
  }\label{fig:pool3oscill}
\end{figure}

\begin{figure}
\centering
\includegraphics[width=0.85\textwidth]{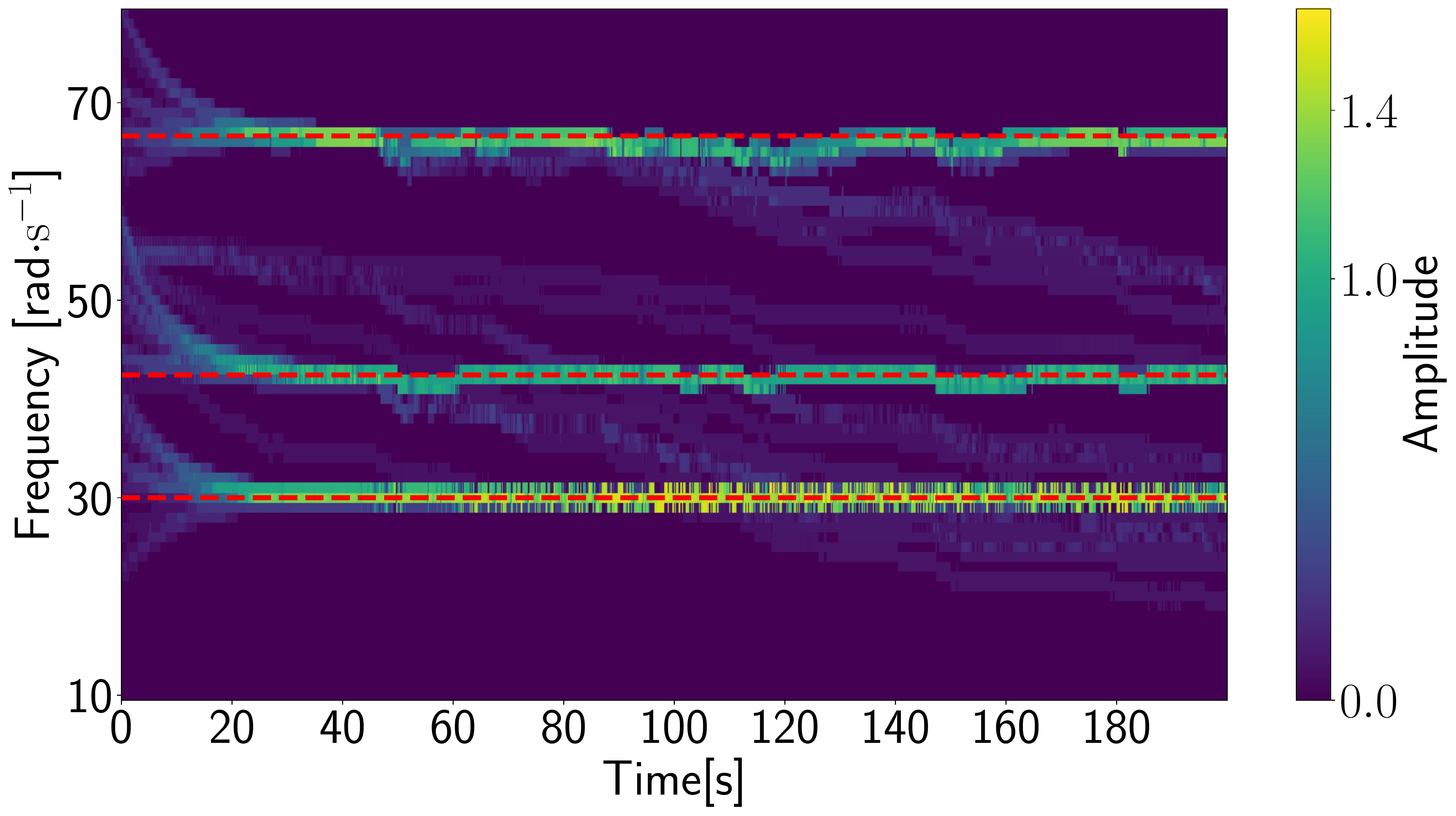}\\
\includegraphics[width=.46\textwidth]{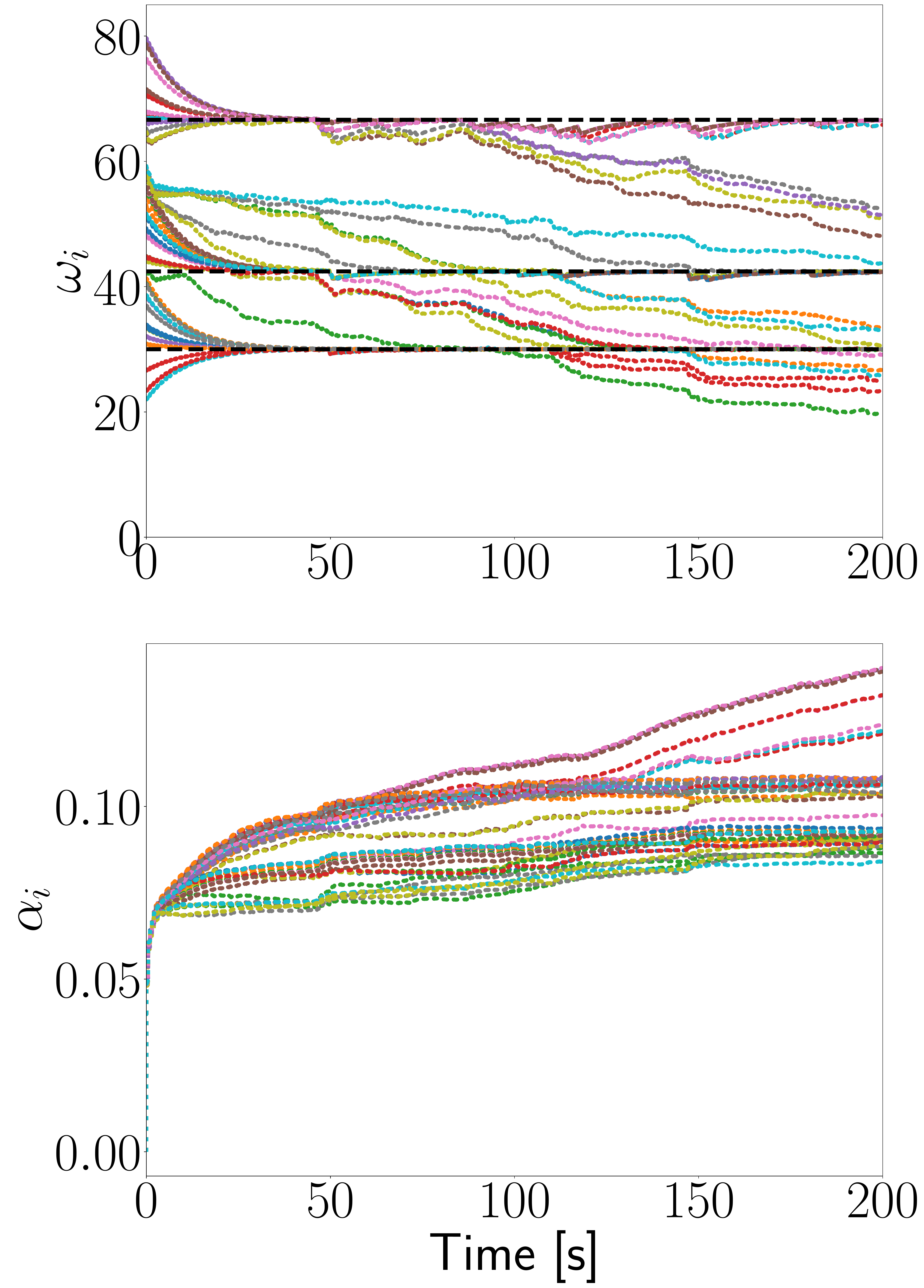}
\includegraphics[width=.46\textwidth]{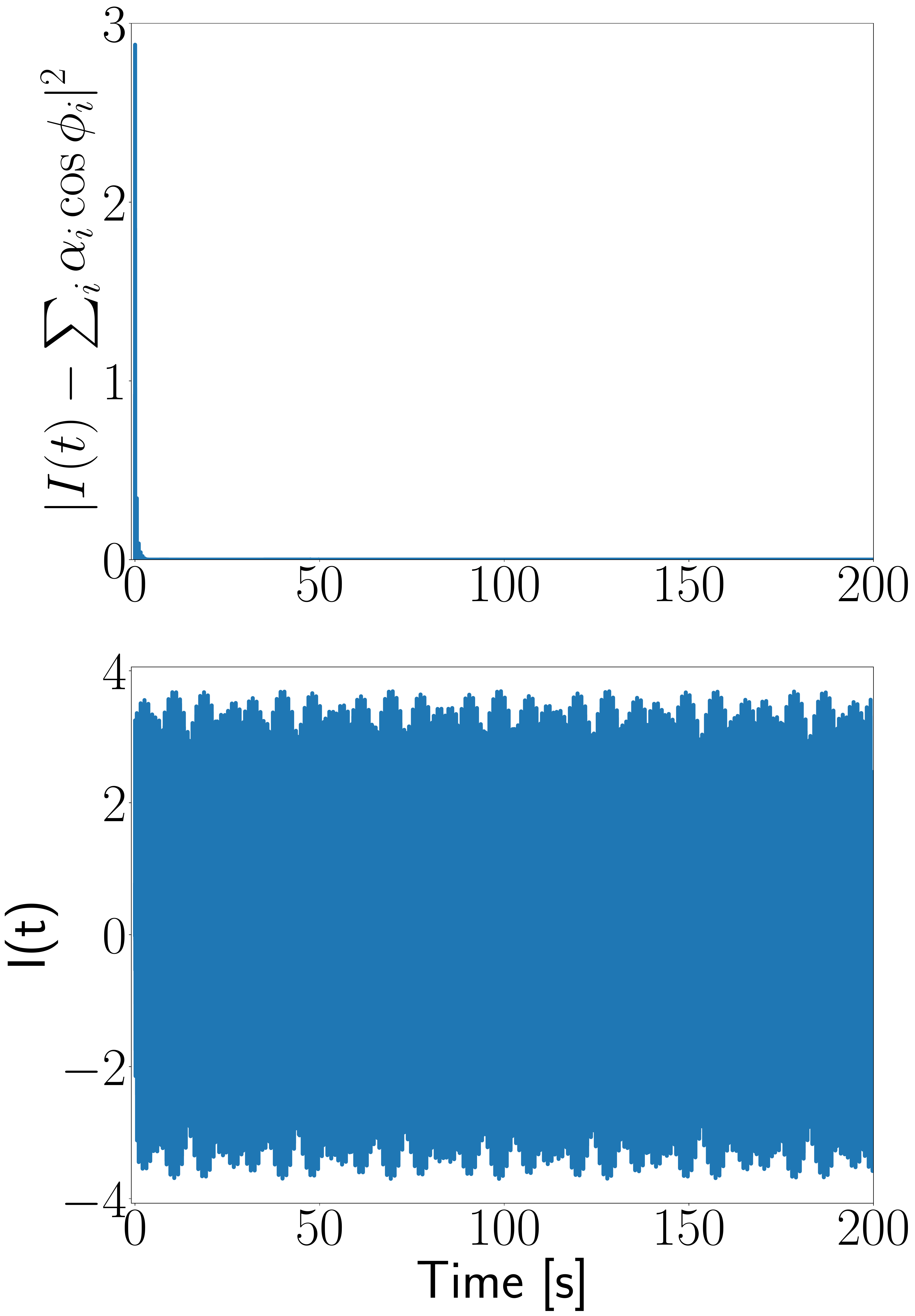}
\caption{Pool of oscillators ($N=50$) with input $I(t) = 1.3\cos(30t) + \cos(30\sqrt{2}t) + 1.4\cos(\frac{30\pi}{\sqrt{2}}t)$. We used $K=10^4$, $\lambda=0.1$ and $\eta=1$ in the simulations. Top: amplitude weighted frequency distribution (cf. text for details) as a function of time. The red dashed lines correspond to the frequencies present in the input. Bottom left: evolution of $\omega_i$ and $\alpha_i$. Bottom right: input signal $I(t)$
and squared error between input and output.}\label{fig:pool50oscill}
\end{figure}

\begin{figure}
\centering
\includegraphics[width=0.85\textwidth]{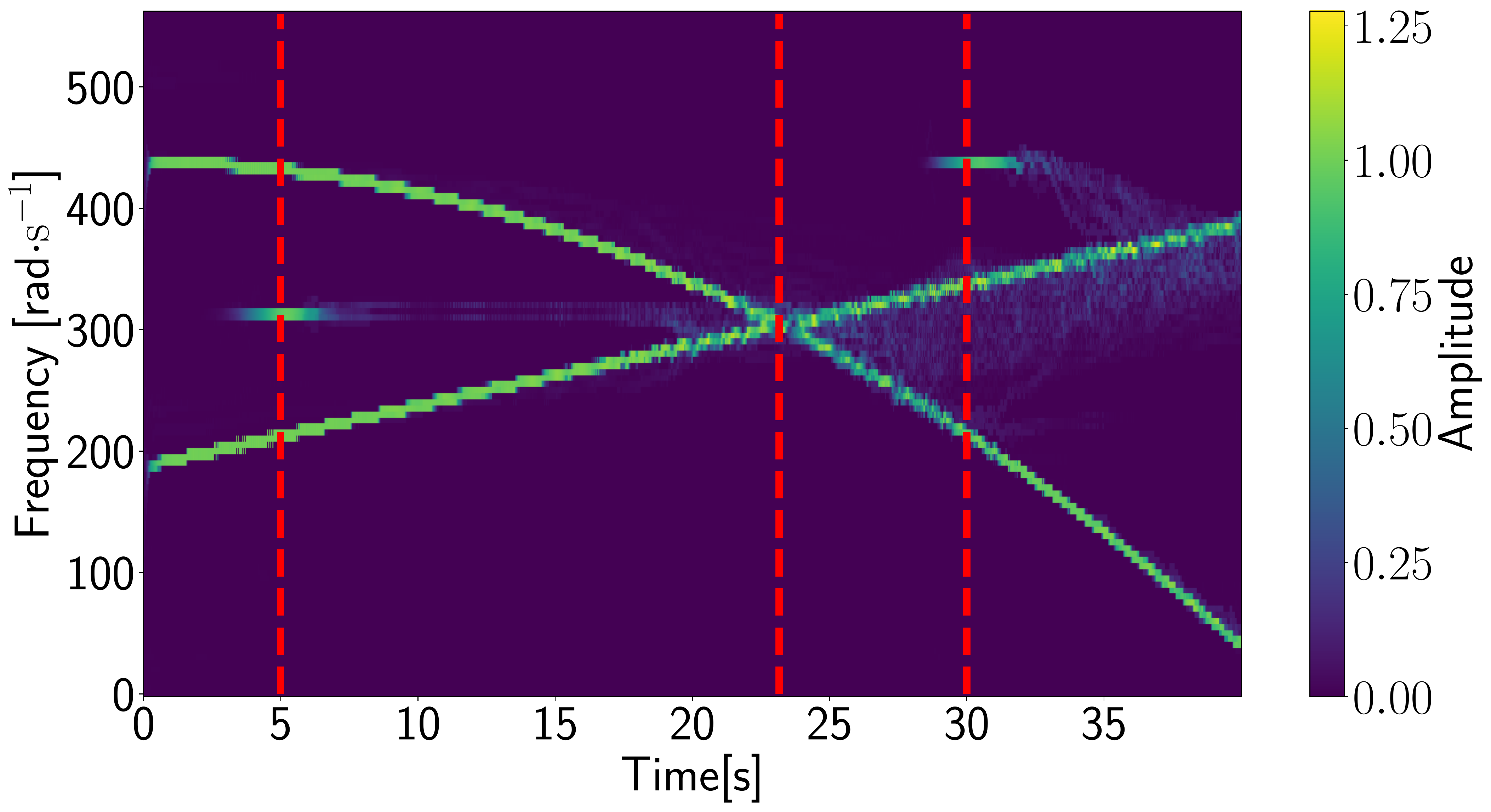}\\
\includegraphics[width=.46\textwidth]{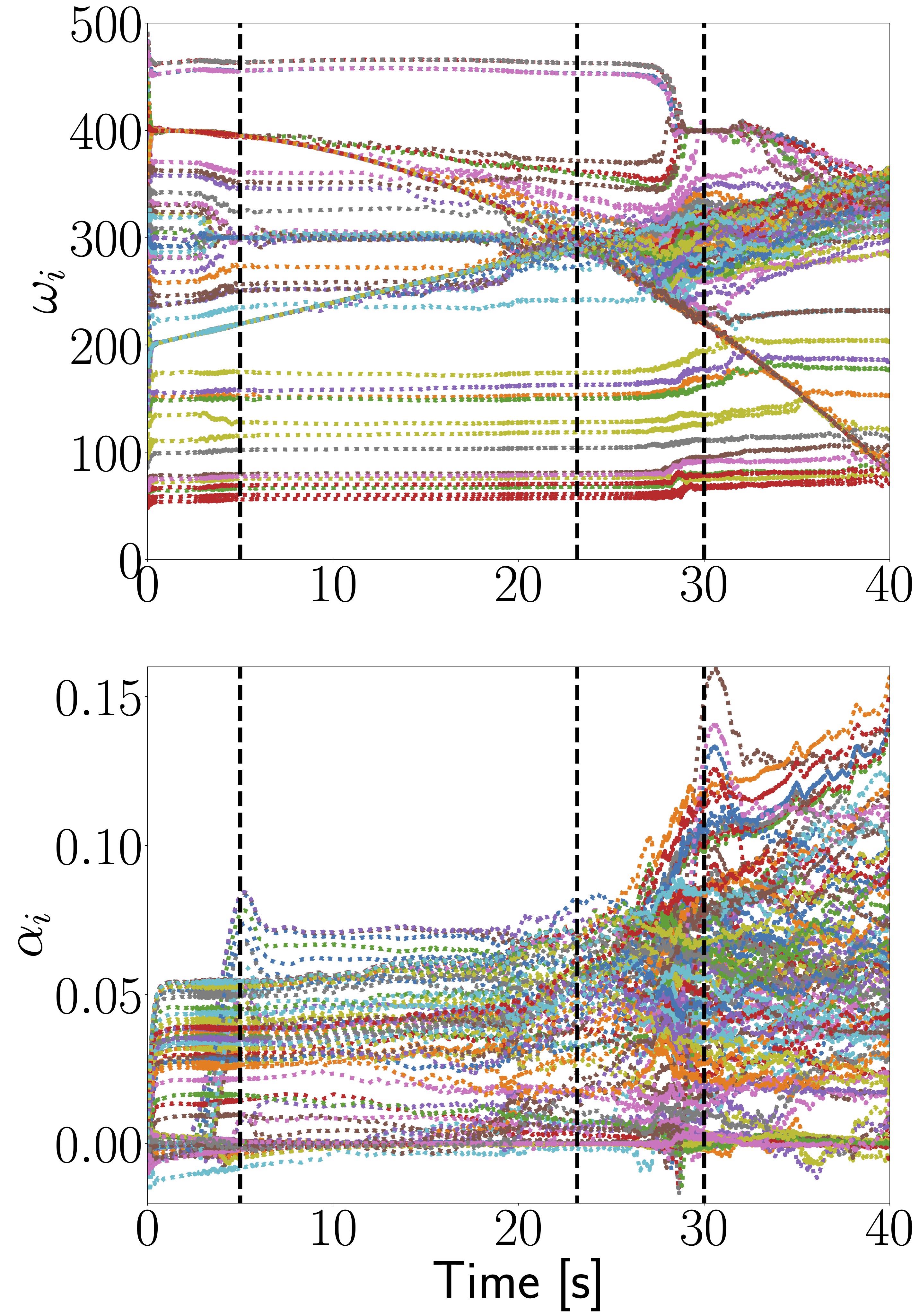}
\includegraphics[width=.46\textwidth]{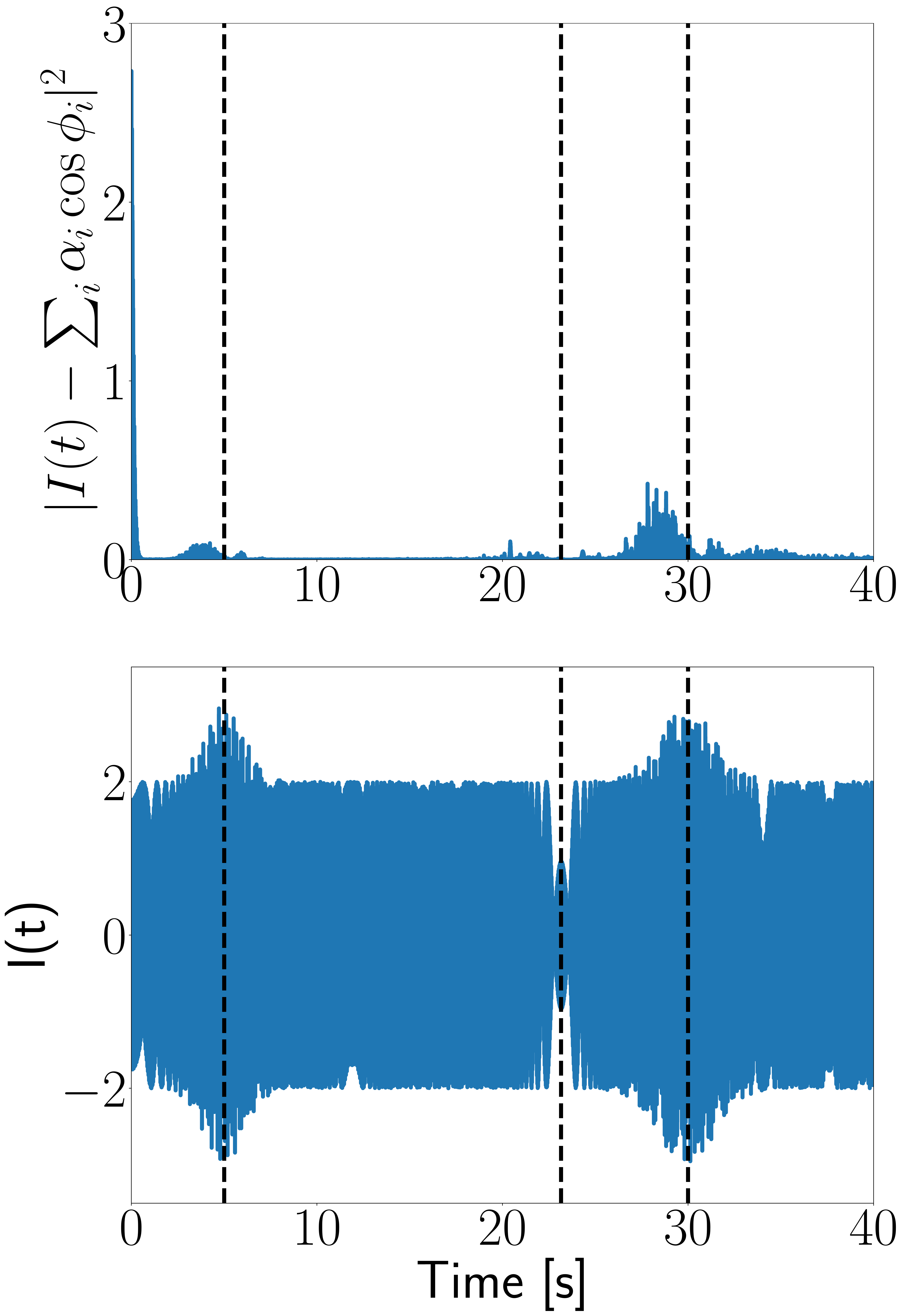}
\caption{Pool of oscillators with time varying spectra input (cf. \cref{sec:time_varying_spectra} for more details). We used $K=100$, $\lambda=10$ and $\eta=0.5$ in the simulations. Top: amplitude weighted frequency distribution as a function of time. Bottom left: evolution of $\omega_i$ and $\alpha_i$. Bottom right: input signal $I(t)$
and squared error between input and output. The vertical dashed lines show the important events in time: maximum of the Gaussians and crossing of the chirps.}\label{fig:time_var}
\end{figure}

\subsection{Numerical experiments}
We now present numerical experiments to illustrate our findings and illustrate some of the capabilities and limits of the system.
\subsubsection{Discrete spectra}
This example shows the typical behavior of a network of oscillators when the input signal has a discrete frequency spectrum. We use $I(t) = 1.3\cos(30t) + \cos(30\sqrt{2}t) + 1.4\cos(\frac{30\pi}{\sqrt{2}}t)$, already used in \Cref{sec:aperioc}.
We showed previously that the frequency of a single oscillator in open-loop would not converge in that case. In contrast, with the feedback loop and amplitude adaptation, the oscillators systematically adapt their states such that $I(t)\simeq \sum_i \alpha_i \cos\phi_i$.

\Cref{fig:pool3oscill} presents results with $N=3$ oscillators (i.e. the minimum number to reconstruct the input frequency spectrum) for different values of $K$. For small K, frequency adaptation becomes exponential when getting close to the input frequency (i.e. when entering the synchronization region - as discussed in the previous section). 
After frequency convergence, the corresponding amplitude is adapted.
Interestingly the green $\omega_i$ crosses the frequencies already taken by the other oscillators to adapt to the remaining frequency.
Eventually all the frequencies and amplitudes converge to the expected values and $I(t) = \sum_i \alpha_i \cos\phi_i$. We observe similar results for $K=100$, except that convergence is exponential from the beginning, leading to faster convergence.
When $K=10000$ (i.e. strong coupling case) the situation is different. Only one $\omega_i$ converges to one of the frequency component of the input. The other $\omega_i$ do not seem to converge to any specific value. On the other hand, the amplitude corresponding to the converged frequency does not seem to converge to a specific value while the two other amplitudes do converge, but not to one of the amplitudes associated to one of the cosine of the input. Nevertheless, the output of the pool of oscillator perfectly reconstructs the input signal very quickly. It is interesting to see that the feedback loop enables the pool of oscillator to reconstruct the signal albeit the exact frequency components in the input are not recovered when the coupling becomes too high.

\Cref{fig:pool50oscill} shows results with $N=50$ oscillators for large coupling. 
The upper graph shows the frequency distribution of the oscillators, i.e. the distribution of $\omega_i$ weighted by the respective $\alpha_i$ while taking into account the oscillator phases. We data presented in this figure is computed as follows. At each time $t$, the frequency spectrum is discretized into frequency bins (1 rad/s in this case) and each $\omega_i(t)$ is associated to a bin. The amplitude
associated to a bin of frequency $\psi$ is computed as $\max_{\bar{t}} ( \sum_i \alpha_i(t) \cos(\psi \bar{t} + \phi_i(t)))$, where $t$ is fixed and $\bar{t}$ is varied to cover at least one period of oscillation. This representation gives the same information as a spectrogram resulting from a windowed Fourier transform.

The frequencies and amplitudes adapt such that
the output of the network reproduces the input very quickly.
Initially, frequencies $\omega_i$ are all attracted to the frequency components present in the input. However frequencies do not seem to all converge. We notice that the overall amplitudes associated to each frequency component, i.e. the contribution of all the oscillators close to this frequency, match well those of the input (visible in the top graph).

Both simulations show that the pool of oscillator can well reproduce a non-periodic input with discrete spectrum whereas a single oscillator without feedback would not converge to any frequency. However, we notice that not all frequencies $\omega_i$ converge to the input frequency components. This behavior suggests that the loss of hyperbolicity for candidate invariant manifolds when $I(t) = \sum_i \alpha_i \cos\phi_i$ might play an important role in describing the complete dynamics as the network dynamics seems to converge to it. The dynamics appears to change qualitatively as coupling increases, which is in contrast to the single oscillator case where convergence happens for any value of $K$. %It might also become important to study the smaller coupling case for the network of oscillators.

\paragraph{Time-varying spectra}\label{sec:time_varying_spectra}
Finally, we illustrate the capabilities of the system for moderate coupling to track a time varying spectrum, with appearing and disappearing frequency components, demonstrating its generic frequency analysis capabilities.
In this example, the input is composed of one linear chirp $\sin(200t + 2t^2)$, one quadratic chirp $\sin(400t - \frac{t^3}{15})$, and two frequency modulated Gaussians: $\sin(300t)\exp^{-\frac{(t-5)^2}{2.5}}$ and $\sin(400t)\exp^{-\frac{(t-30)^2}{5}}$. We use $N=100$ oscillators. The results are shown in \cref{fig:time_var}.
 We see that the system is able to track the chirps and to appropriately locate the Gaussians. All the important features of the signal are clearly visible.
We also notice that the error between the system output and the input is almost always $0$, except when a new component appears (the Gaussian) or when the quadratic chirp becomes too fast, but still the match is very good.
The time evolution of $\omega_i$ and $\alpha_i$ shows that oscillators that are not used to encode the chirps are recruited when an event appears (e.g. the Gaussians). We also notice appearing and disappearing clusters of frequencies and amplitudes representing the different signals.

\section{Conclusion}
We analyzed the geometric structure of frequency adaptation
for an adaptive frequency phase oscillator with strong coupling.
We characterized the existence of invariant slow manifolds and demonstrated
that the frequency adaptation mechanism resulted from 
the alternation of slow and fast dynamics, regulated by sign changes of an input signal. 
The slow-fast dynamics described in this paper is rather unique in that regard.
Our analysis enabled to extend such systems to set the exponential convergence rate  $\lambda$.
A discrete map summarizing the slow-fast dynamics allowed to characterize important features of the system, such as its convergence rate or predicting that for some non-periodic signals frequency adaptation would not converge.

We have also analyzed the case of a network of adaptive frequency oscillators with amplitude adaptation. Interestingly, the slow manifolds characterized in the simple oscillator case persist and the feedback loop leads to the appearance of a novel type of slow manifolds for the single oscillator case. When several oscillators are used, the novel critical manifold is not hyperbolic nor nil-potent.
While numerical simulations show that the system converges to $I(t) = \sum \alpha_i \cos \phi_i$, whether there exists slow invariant manifolds of this type remains an open question. Numerical simulations further showed the ability of the system to track complex signals with time-varying frequency components.

To the best of our knowledge, previous work describing 
adaptive frequency oscillators (e.g. second order oscillators or oscillators with intertial effects \cite{acebron98, acebron_synchronization_2000, tanaka_first_1997, taylor_spontaneous_2010}) need an explicit representation of either the phase of other oscillators, or the input signal’s period or frequency. This is in contrast with the system we described which can extract the frequency of arbitrary inputs. 
From this viewpoint, the mechanism described is potentially more practical in engineering applications where input signals are not known in advance and can be noisy and time-varying.
We believe that the results presented in this paper will further help design real applications beyond existing ones in robotics, control and estimation applications \cite{buchli08b, gams08b, righetti06, Petric:2011dr, ronsse11}, further help understand their fundamental limitations and facilitate their usage in real-world settings.

\bibliographystyle{siamplain}
\bibliography{pooloscill}

\begin{thebibliography}{10}

\bibitem{codelink}
\url{https://github.com/righetti/AFOs}.

\bibitem{acebron98}
{\sc J.~Acebron and R.~Spigler}, {\em Adaptive frequency model for
  phase-frequency synchronization in large populations of globally coupled
  nonlinear oscillators}, Physical Review Letters, 81 (1998), pp.~2229--2232.

\bibitem{acebron_synchronization_2000}
{\sc J.~A. Acebrón, L.~L. Bonilla, and R.~Spigler}, {\em Synchronization in
  populations of globally coupled oscillators with inertial effects}, Physical
  Review E, 62 (2000), pp.~3437--3454,
  \url{https://doi.org/10.1103/PhysRevE.62.3437},
  \url{https://link.aps.org/doi/10.1103/PhysRevE.62.3437} (accessed
  2021-03-23).

\bibitem{ahmadi09}
{\sc A.~Ahmadi, E.~Mangieri, K.~Maharatna, and M.~Zwolinski}, {\em Physical
  realizable circuit structure for adaptive frequency hopf oscillator}, in
  NEWCAS-TAISA, Toulouse, France, July 2009.

\bibitem{borisyuk01}
{\sc R.~Borisyuk, M.~Denham, F.~Hoppensteadt, Y.~Kazanovich, and
  O.~Vinogradova}, {\em Oscillatory model of novelty detection}, Network:
  Computation in neural systems, 12 (2001), pp.~1--20.

\bibitem{borisyuk_oscillatory_2004}
{\sc R.~M. Borisyuk and Y.~B. Kazanovich}, {\em Oscillatory model of
  attention-guided object selection and novelty detection}, Neural Networks, 17
  (2004), pp.~899--915, \url{https://doi.org/10.1016/j.neunet.2004.03.005},
  \url{http://linkinghub.elsevier.com/retrieve/pii/S089360800400070X}.
\newblock Publisher: Elsevier Ltd.

\bibitem{buchli06b}
{\sc J.~Buchli, F.~Iida, and A.~Ijspeert}, {\em Finding resonance: Adaptive
  frequency oscillators for dynamic legged locomotion}, in Proceedings of the
  {IEEE/RSJ} International Conference on Intelligent Robots and Systems (IROS),
  IEEE, 2006, pp.~3903--3909.

\bibitem{buchli04b}
{\sc J.~Buchli and A.~Ijspeert}, {\em A simple, adaptive locomotion
  toy-system}, in From Animals to Animats 8. Proceedings of the Eighth
  International Conference on the {S}imulation of {A}daptive {B}ehavior
  (SAB'04), S.~Schaal, A.~Ijspeert, A.~Billard, S.~Vijayakumar, J.~Hallam, and
  J.~Meyer, eds., MIT Press, 2004, pp.~153--162.

\bibitem{buchli08b}
{\sc J.~Buchli and A.~Ijspeert}, {\em Self-organized adaptive legged locomotion
  in a compliant quadruped robot}, Autonomous Robots, 25 (2008), pp.~331--347,
  \url{10.1007/s10514-008-9099-2}.

\bibitem{buchli05}
{\sc J.~Buchli, L.~Righetti, and A.~Ijspeert}, {\em A dynamical systems
  approach to learning: a frequency-adaptive hopper robot}, in Proceedings of
  the {VIII}th European Conference on Artificial Life {ECAL} 2005, Lecture
  Notes in Artificial Intelligence, Springer Verlag, 2005, pp.~210--220.

\bibitem{buchli08}
{\sc J.~Buchli, L.~Righetti, and A.~Ijspeert}, {\em Frequency analysis with a
  nonlinear dynamical system}, Physica D,  (2008),
  \url{http://dx.doi.org/10.1016/j.physd.2008.01.014}.

\bibitem{ermentrout91}
{\sc B.~Ermentrout}, {\em An adaptive model for synchrony in the firefly
  pteroptyx malaccae}, Journal of mathematical biology, 29 (1991),
  pp.~571--585.

\bibitem{fenischel79}
{\sc N.~Fenichel}, {\em Geometric singular perturbation theory for ordinary
  differential equations}, Journal of Differential Equations, 31 (1979),
  pp.~53--98.

\bibitem{gams08b}
{\sc A.~Gams, S.~Degallier, A.~Ijspeert, and J.~Lenar\v{c}i\v{c}}, {\em
  Dynamical system for learning the waveform and frequency of periodic signals
  \& application to drumming}, in Proceedings of the 17th International
  Workshop on Robotics in Alpe-Adria-Danube Region (RAAD2008), 2008.

\bibitem{gams09}
{\sc A.~Gams, A.~Ijspeert, S.~Schaal, and J.~Lenarcic}, {\em On-line learning
  and modulation of periodic movements with nonlinear dynamical systems},
  Autonomous Robots, 27 (2009), pp.~3--23.

\bibitem{ijspeert08}
{\sc A.~Ijspeert}, {\em Central pattern generators for locomotion control in
  animals and robots: a review}, Neural Networks, 21 (2008), pp.~642–--653.

\bibitem{Jardon-Kojakhmetov_Kuehn_2019}
{\sc H.~Jardon-Kojakhmetov and C.~Kuehn}, {\em A survey on the blow-up method
  for fast-slow systems}, arXiv:1901.01402 [math],  (2019),
  \url{http://arxiv.org/abs/1901.01402}.
\newblock arXiv: 1901.01402.

\bibitem{jones1609geometric}
{\sc C.~Jones}, {\em {Geometric singular perturbation theory }}, in Dynamical
  Systems Lectures Given at the nd Session of the Centro Internazionale
  Matematico Estivo C.I.M.E. held in Montecatini Terme, Italy, June --,,
  Springer Berlin Heidelberg, Berlin, Heidelberg, 1995, pp.~44--118.

\bibitem{nachstedt2017}
{\sc T.~Nachstedt, C.~Tetzlaff, and P.~Manoonpong}, {\em Fast dynamical
  coupling enhances frequency adaptation of oscillators for robotic locomotion
  control}, Frontiers in Neurorobotics, 11 (2017), p.~14,
  \url{https://doi.org/10.3389/fnbot.2017.00014},
  \url{https://www.frontiersin.org/article/10.3389/fnbot.2017.00014}.

\bibitem{nakanishi03}
{\sc J.~Nakanishi, J.~Morimoto, G.~Endo, G.~Cheng, S.~Schaal, and M.~Kawato},
  {\em Learning from demonstration and adaptation of locomotion with dynamical
  movement primitives}, Robotics and Autonomous Systems, 47 (2003), pp.~79--91.

\bibitem{nishii99}
{\sc J.~Nishii}, {\em Learning model for coupled neural oscillators}, Network:
  Computation in neural systems, 10 (1999), pp.~213--226.

\bibitem{Petric:2011dr}
{\sc T.~Petric, A.~Gams, A.~Ijspeert, and L.~{\v Z}lajpah}, {\em {On-line
  frequency adaptation and movement imitation for rhythmic robotic tasks}}, The
  International Journal of Robotics Research, 30 (2011), pp.~1775--1788.

\bibitem{righetti06}
{\sc L.~Righetti and I.~A.J.}, {\em Programmable central pattern generators: an
  application to biped locomotion control}, in Proceedings of the 2006 IEEE
  International Conference on Robotics and Automation, 2006.

\bibitem{righetti05b}
{\sc L.~Righetti, J.~Buchli, and A.~Ijspeert}, {\em From dynamic hebbian
  learning for oscillators to adaptive central pattern generators}, in
  Proceedings of 3rd International Symposium on Adaptive Motion in Animals and
  Machines -- AMAM 2005, Verlag ISLE, Ilmenau, 2005.
\newblock Full paper on CD.

\bibitem{righetti06b}
{\sc L.~Righetti, J.~Buchli, and A.~Ijspeert}, {\em Dynamic hebbian learning in
  adaptive frequency oscillators}, Physica D, 216 (2006), pp.~269--281,
  \url{http://dx.doi.org/10.1016/j.physd.2006.02.009}.

\bibitem{Rodriguez:2014ey}
{\sc J.~Rodriguez and M.~O. Hongler}, {\em {Networks of Self-Adaptive Dynamical
  Systems}}, IMA Journal of Applied Mathematics, 79 (2014), pp.~201--240.

\bibitem{Ronsse:im}
{\sc R.~Ronsse, S.~De~Rossi, N.~Vitiello, T.~Lenzi, M.~Carrozza, and
  A.~Ijspeert}, {\em {Real-Time Estimate of Velocity and Acceleration of
  Quasi-Periodic Signals Using Adaptive Oscillators}}, IEEE Transactions on
  Robotics, 29 (2013), pp.~783--791.

\bibitem{ronsse11}
{\sc R.~Ronsse, N.~Vitiello, T.~Lenzi, J.~van~den Kieboom, M.~Carrozza, and
  A.~Ijspeert}, {\em {Human--Robot Synchrony: Flexible Assistance Using
  Adaptive Oscillators}}, Biomedical Engineering, IEEE Transactions on, 58
  (2011), pp.~1001--1012.

\bibitem{tanaka_first_1997}
{\sc H.-A. Tanaka, A.~J. Lichtenberg, and S.~Oishi}, {\em First {Order} {Phase}
  {Transition} {Resulting} from {Finite} {Inertia} in {Coupled} {Oscillator}
  {Systems}}, Physical Review Letters, 78 (1997), pp.~2104--2107,
  \url{https://doi.org/10.1103/PhysRevLett.78.2104},
  \url{https://link.aps.org/doi/10.1103/PhysRevLett.78.2104} (accessed
  2021-03-23).

\bibitem{taylor_spontaneous_2010}
{\sc D.~Taylor, E.~Ott, and J.~G. Restrepo}, {\em Spontaneous synchronization
  of coupled oscillator systems with frequency adaptation}, Physical Review E,
  81 (2010), p.~046214, \url{https://doi.org/10.1103/PhysRevE.81.046214},
  \url{https://link.aps.org/doi/10.1103/PhysRevE.81.046214} (accessed
  2021-03-23).

\end{thebibliography}
\end{document}